\documentclass[preprint,12pt]{elsarticle}

\usepackage{amsmath}
\usepackage{amsthm}
\usepackage{amssymb}
\usepackage{mathrsfs}
\usepackage{amsthm}
\usepackage{color}
\usepackage{graphicx}
\usepackage{subfig} 
\usepackage{xspace}
\usepackage{float}
\theoremstyle{plain}
\newtheorem{theorem}{Theorem}[section]
\newtheorem{lemma}[theorem]{Lemma}
\newtheorem{assumption}[theorem]{Assumption}
\newtheorem{proposition}[theorem]{Proposition}
\newtheorem{corollary}[theorem]{Corollary}

\theoremstyle{remark}
\newtheorem{remark}[theorem]{Remark}
\newtheorem{example}[theorem]{Example}

\journal{Mathematics and Computers in Simulation
}
\begin{document}

\begin{frontmatter}
\title{Sensitivity Analysis for Mean-Field SDEs with Jumps by Malliavin Calculus: Chaos Expansion Approach}


\author{Samaneh Sojudi and Mahdieh Tahmasebi}
\address{Department of Applied Mathematics, Faculty of Mathematical Sciences, Tarbiat Modares University, P.O. Box 14115-134, Tehran,
Iran} 

\begin{abstract}
In this paper, we describe an explicit extension formula in sensitivity analysis regarding the Malliavin weight for jump-diffusion
mean-field stochastic differential equations whose local Lipschitz drift coefficients are influenced by the product of the solution and its law. We state that these extended equations have unique Malliavin differentiable solutions in the Wiener-Poisson space and establish the sensitivity analysis of path-dependent discontinuous payoff functions. It will be realized after finding a relation between the stochastic flow of the solutions and their derivatives. The Malliavin derivatives are defined in a chaos approach that the chain rule is not held. The convergence of the Euler method to approximate Delta Greek is proved. The simulation experiment 
illustrates our results to compute the Delta, in the context of financial mathematics, and demonstrates that the Malliavin Monte-Carlo computations applied in our formula
are more efficient than using the finite difference method directly.
\end{abstract}


\begin{keyword}
                                    

Mean-field SDEs, Malliavin calculus, Delta, Jump diffusion processes
\end{keyword}

\end{frontmatter}


\section{Introduction}  

\noindent In 1999, Fourni\'e et al. \cite{fournie1999applications} introduced the Greeks criteria using Malliavin calculus on the Wiener space. They demonstrated that the convergence rate of their method is more efficient than the finite difference method, particularly for discontinuous payoff functions. This innovative technique has since been applied by various researchers across different models. Here are some examples of what has been done. In 2004, Youssef Al-Khatib and Nickolas Perfila in \cite{el2004computations} extended this approach to calculate the Greeks in financial markets modeled by discontinuous processes with Poisson jump times and random jump sizes. Similarly, Davis and Johanson \cite{davis2006malliavin} computed these quantities for jump-diffusion processes satisfying separability conditions. Bavouzet-Morel and Messaoud \cite{bavouzet2006computation} analyzed the sensitivity of European and Asian options in jump-diffusion models using Malliavin calculus, showing that the variance of the Brownian-jump Malliavin Monte Carlo method is significantly smaller than the standard Brownian Monte Carlo method.  

\noindent Subsequent works have further advanced the use of Malliavin calculus in jump-diffusion models. For example, Forster et al. \cite{forster2009absolutely} investigated Greeks on spaces generated by independent Wiener and Poisson processes under specific conditions on the jump part in finite and infinite dimensions. Each of these approaches has been only investigated and covered by a subgroup of Lévy processes until Petrou 
\cite{petrou2008malliavin}
 considered a general class of  Lévy processes. 
 Petrou's approach is based on the fact that every Lévy process can be decomposed into an integral of a Wiener process and an integral with respect to a Poisson random measure. With the chaos expansion in hand, he introduced derivatives and their dual operators for processes in the Wiener-Poisson space. By his generalization, he can bypass the separability condition and the chain rule needed in the sensitivity analysis and apply integration by parts formula. In 
\cite{benth2011robustness},
 Benth et, al. introduced an innovative approach. Benth states that the Malliavin calculus can be used in cases without continuous martingale components in jump-diffusion dynamics. He explained his idea in such a way that small jumps can be approximated by a continuous martingale with appropriate variance. He also states that the Delta obtained based on this approximation is close to the real Delta. This method was later extended by Khedher \cite{khedher2012computation} to multidimensional diffusion Lévy processes and applied it to Delta calculation for stochastic models with volatility. In 2023, Hude and Rosendorf \cite{hudde2023european} calculated Greeks for exponential models with vanishing Brownian components, demonstrating that the Malliavin Monte Carlo method yields a significantly lower variance than traditional methods.

\noindent Given the efficiency and effectiveness of Malliavin calculus, it has gained significant attention in financial mathematics  (see for instance, \cite{yilmaz2018computation}, \cite{benth2021sensitivity}, \cite{fan2022bismut}, \cite{coffie2023sensitivity}). Recently, this subject has been explored specifically for mean-field stochastic differential equations (SDEs).  

\noindent  The mean-field SDEs, or McKean-Vlasov SDEs or distribution-dependent SDEs, were first studied by Kac
\cite{kac1956foundations}
in 1956 and it was followed by McKean 
\cite{mckean1966class}
in 1966 on stochastic systems with a large number of interacting particles and then by Graham 
\cite{graham1992mckean}
in 1992 on McKean-Vlasov Itô-Skorohod equations and nonlinear diffusions with discrete jump sets. 
Mckean-Vlasov SDEs allow us to describe the collective behavior of a large number of particles in statistical physics and analysis the large-scale social interactions, especially within the framework of mean-field games and financial derivatives pricing, see e.g.
(\cite{carmona2018probabilistic}, \cite{bush2011stochastic}, \cite{hambly2019spde}, \cite{dawson1995stochastic})
and references therein.

\noindent In this manuscript, we undertake a rigorous sensitivity analysis of mean-field models through the sophisticated methodology of Malliavin calculus. Over recent years, notable advancements have been made in this area. For instance, in 2018, David Ba\~nos \cite{banos2018bismut} investigated the Bismut formula for multi-dimensional mean-field SDEs with multiplicative Gaussian noise, emphasizing smooth drift and volatility coefficients. Similarly, the implementation of the Bismut formula for mean-field SDEs under additive noise was detailed in
\cite{bao2021bismut}.
Song and Wang 
\cite{song2022regularity}, 
examined the regularity properties of multidimensional distribution-dependent SDEs driven by a pure jump process, deriving a Bismut-type derivative formula based on stochastic gradient and Malliavin derivative operators (as outlined in Chapter 11 of \cite{nualart2018introduction}), which adheres to the chain rule. Moreover, an extension of the Bismut-Elworthy-Li formula for semi-linear mean-field SDEs driven by fractional Brownian motion was presented in
\cite{tahmasebi2022bismut}, with applications demonstrated in the pricing of variance swaps and other derivatives.\\
\noindent To illustrate the application of mean-field SDEs with jumps, especially, we can reference \cite{agram2024optimal} and \cite{agram2024deep}. In \cite{agram2024optimal}, the authors explore the optimal stopping problem for conditional McKean-Vlasov jump diffusions and illustrate their results to find the optimal time to sell a property like a house modeled in this type of SDEs. In \cite{agram2024deep}, the authors concentrate on optimizing the control of mean-field jump diffusions by implementing of deep learning methodologies. This study examines multi-particle jump-diffusion dynamics in the linear quadratic conditional mean-field (LQCMF) framework. They introduce a deep learning algorithm that integrates neural networks with path signatures to estimate the conditional expectations.  As a practical scenario, including LQCMF, interbank systemic risk (borrowing/lending) were discussed in the numerical outcomes.\\
\noindent The dependence of the coefficients on both the solution itself and its probabilistic law inherently introduces significant complexities into the analysis of these equations. These challenges are further compounded by the assumptions and conditions imposed on the coefficients, such as smoothness, boundedness, or Lipschitz continuity, which are often required to ensure the well-posedness and regularity of the solutions. Such constraints not only complicate the mathematical treatment of the equations but also limit the generality of the models, making the analysis highly intricate and technically demanding. Li and Hui in \cite{li2016weak}  explore stochastic differential equations defined by a drift coefficient \(b(s, X_{t \wedge s}, Q_{X_s})\) with globally Lipschitz condition in 1-Wasserstein distance. 
For locally Lipschitz conditions, the existing literature remains relatively limited. To the best of our knowledge, the first local condition was introduced in the pioneering and elegant work of Carmona and Delarue (2015).
In \cite{bauer2018strong}, Bauer et al. investigate the existence and uniqueness of strong solutions for mean-field SDEs characterized by irregular drift coefficients. For the case of a globally Lipschitz continuous diffusion coefficient and a one-sided Lipschitz continuous drift coefficient, Wang (2018) demonstrated the well-posedness of strong solutions. Further advancing the discourse on mean-field SDEs, Mishura and Veretennikov \cite{mishura2020existence} analyze special mean-field SDEs as the form
\[  
X_t = X_0 + \int_0^t \int_{\mathbb{R}^d} b(s, X_s, y) P_{X_s}(dy) ds + \int_0^t \int_{\mathbb{R}^d} \sigma(s, X_s, y) P_{X_s}(dy) dB_s, \quad t \geq 0,  
\] \\
\noindent where coefficient $b$ satisfies the assumption that no regularity of the drift is required for weak and strong uniqueness, relying on the analysis of the total variation metric.\\
\noindent Erny in \cite{erny2022well} established the well-posedness and propagation of chaos for McKean-Vlasov equations with jumps and locally Lipschitz coefficients. The analysis in their work relies on a Picard iteration scheme and truncation arguments, addressing technical challenges arising from the locally Lipschitz nature of the coefficients. Notably, their results hold under a growth condition of the form  
\begin{align} \label{intro1} 
    |b(t, x, m)| \leq C(1 + |x|),  
\end{align}  
where \( C > 0 \), \( x \in \mathbb{R} \), and \( m \in \mathcal{P}_1(\mathbb{R}) \), ensuring the drift coefficient remains well-behaved for large values of \( x \). This growth condition plays a crucial role in ensuring the strong well-posedness of the equations and we are interested in overcoming this limitation in our work by some weaker condition on the drift.
Additionally, Liu et al. \cite{liu2023onsager}, in 2023, proved the well-posedness of McKean-Vlasov SDEs with the drift coefficient being locally Lipschitz and the diffusion coefficient being Hölder continuous with respect to the state variable.
In \cite{ning2024one}, different from many works, Ning et al. consider the existence and the uniqueness of solution under much weaker conditions on the coefficients, with $b$ being locally Lipschitz and $\sigma$ being locally Hölder continuous in the state variable
\begin{align}  \label{intro2}
    |b(\omega, s, x, \mu)& - b(\omega, s, x', \mu')| \notag\\   
    &\leq \Big(C \ln(e + |x| + |x'|) + \mu(|\cdot|) + \mu'(|\cdot|)\Big) \Big[|x - x'| + W_1(\mu, \mu')\Big],  
\end{align}
for any $x, x' \in \mathbb{R}$. The Wasserstein space of order $p$, for $p \geq 1$, is defined as  
\begin{align*}  
    \mathcal{P}_p(\mathbb{R}) := \left\{ \nu \in \mathcal{P}(\mathbb{R}) : \int_{\mathbb{R}} |x|^p \nu(dx) < \infty \right\},  
\end{align*}  
and the $p$-th order Wasserstein distance, as
\begin{align*}  
    W_p(\nu, \nu') := \inf_{\pi \in \mathfrak{L}(\nu, \nu')}   
    \left( \int_{\mathbb{R} \times \mathbb{R}} |x - y|^p d\pi(x, y) \right)^{\frac{1}{p}}, 
\end{align*}
where $\mathfrak{L}(\nu, \nu')$ denotes the set of all couplings of $\nu$ and $\nu'$. 
\noindent Jianyu Chen and colleagues, in \cite{chen2024correspondence}, investigated the indirect approximation of the most probable transition pathway for stochastic interacting particle systems and their mean field limit (McKean-Vlasov SDEs). Their approach utilizes the Onsager-Machlup action functional, reformulated as an optimal control problem through the stochastic Pontryagin’s Maximum Principle. In that article, the authors assumed $b : \mathbb{R}^d \times \mathbb{R}^d \to \mathbb{R}^d$ has a decomposition such that $b(t, X_t, y) = h(t, X_t)y$ holds for some function $h(t, X_t)$. In this case, the drift coefficient is the form $h(t, X_t) \mathbb{E}(X_t)$ and satisfies the conditions $\bold{H1- H3}$ in \cite{liu2023onsager} in 2-Wasserstein distance as $h$ is a bounded function. Meanwhile, most stochastic dynamic models do not satisfy these common assumptions, such as $\int_{\mathbb{R}}^{}{xy\mu(dy)}$, and we are interest in considering the latter one, first with the simple linear case of $h$. In our work, we aim to generalize these conditions when the drift coefficient naturally takes the form $h(t, X_t)k(t, \mathbb{E}(X_t))$ when $h$ is not bounded and $k$ is a Lipschitz function and to streamline the concept we assume that $h(t, X_t)=X_t$. To establish the existence and uniqueness of solutions for such equations, we initially focus on this linear form, motivated by \cite{buckdahn2017mean}, we prove the existence and uniqueness of the solution with bounded $p$-moments, $p \geq 2$. We also prove an extension of the Bismut formula under the dynamics of jump-diffusion mean-field SDEs via Malliavin derivatives, as defined in Chapter 10 of \cite{nualart2018introduction}, utilizing chaos expansion, which does not satisfy the chain rule. Furthermore, we demonstrate this sensitivity analysis, particularly, in financial applications like computing the Delta, and we show that it achieves lower variance compared to the finite difference approach for discontinuous payoff functions.\\
\noindent Numerically, recent research has significantly advanced the convergence of Itô-Taylor and Euler schemes for mean-field SDEs, enhancing their applicability in finance and other fields. In \cite{sun2017ito}, numerical methods for mean-field SDEs were explored, extending the fundamental strong convergence theorem and providing numerical examples to validate theoretical results. Building on this foundation, \cite{sun2020explicit}, developed an explicit second-order scheme for decoupled mean-field forward-backward SDEs, demonstrating stability and rigorous error estimates that confirm its second-order accuracy when combined with the weak-order 2.0 Itô-Taylor scheme. Further, \cite{sun2021numerical}, focused on mean-field SDEs with jumps, proposing strong and weak Itô-Taylor schemes with established convergence rates. In \cite{liu2023tamed}, the authors establish the strong convergence of the tamed Euler-Maruyama approximation for McKean–Vlasov SDEs with super-linear drift and Hölder diffusion coefficients. The goal of the paper in \cite{bao2022approximations}, is to approximate two kinds of McKean–Vlasov stochastic differential equations with irregular coefficients via weakly interacting particle systems. The analysis of Euler-Maruyama schemes for approximating McKean-Vlasov SDEs was conducted in \cite{li2023strong}, where strong convergence was shown under weaker local Lipschitz conditions, enhancing the understanding of solution existence and uniqueness. Additionally, \cite{ning2024one}, investigated McKean-Vlasov equations, including stochastic variational inequalities and coupled forward-backward forms, establishing well-posedness and uniqueness of strong solutions under locally Hölder continuous stochastic coefficients.
\noindent In the application section of this article, we demonstrate the convergence of the Euler scheme for mean-field SDEs with jumps, under specific assumptions regarding the drift coefficients. While investigating the convergence of the Euler scheme for Delta by introducing innovative approaches, particularly for barrier options, we build upon the findings in \cite{higham2005convergence}, demonstrating that an Euler–Maruyama discretization yields accurate moment approximations and strong convergence of path-dependent options such as up-and-out barrier options.\\
\noindent The paper is organized as follows: \\
In section 2, we first present a few basic facts about Malliavin calculus for jump processes based on chaos expansion. These facts are used to derive the main results of the paper. In section 3, we introduce mean-field SDEs with the jump and prove the existence and uniqueness of their solutions and some proper important relations for the next sections. The existence of the stochastic flow and its representation by a Skorokhod integrable process are established in section 4. We also find the relation of the flow to the Malliavin derivative in some semi-linear cases. The Malliavin differentiability of path-dependent pay-offs in the form of barrier options is also included in this section. We will show the applications of our results in finding Malliavin weights to compute the Delta in mean-field SDEs with the jump for vanilla options and barrier lookback options in section 5. Section 6 is devoted to proving the convergence of the Euler method to approximate Delta. Finally, we illustrate our results in section 7. To verify efficiency, we will compare our results with the center finite difference method.

\section{Preliminaries}
\noindent In this section we recall some concepts of Malliavin derivatives and Skorohod integral for the jump process based on chaos expansion. More details can be found in the exhaustive references
\cite{nualart2018introduction}, \cite{nunno2008malliavin}.
\\
Suppose 
$\mathit{(\Omega, \mathcal{F}, \mathit{P})}$ is a complete probability space, and $\mathcal{B}({\mathbb{R}}_{0})$ is a $\sigma$-algebra derived from all Borel subspaces of ${\mathbb{R}}_{0}=\mathbb{R}-\{0\}$. The compensated jump measure $\tilde{N}$ is defined by
\begin{equation*}
\tilde{N}\mathit{(dt, dz)} := \mathit{N(dt , dz)} - \mu(\mathit{dz})\mathit{dt},
\end{equation*}
where $\mathit{N(dt, dz)}$ is a Poisson random measure and $\mu$ is the Lévy measure, which is finite measure on $\mathcal{B}({\mathbb{R}}_{0})$.
Let ${\mathit{L}}^2 ((\lambda \times \mu)^{\mathit{n}}) = {\mathit{L}}^2 (([0, \mathit{T}] \times {{\mathbb{R}}_{0}})^{\mathit{n}}), n=1, 2, ..., $ be the space of deterministic real functions $\mathit{f}$ such that
\begin{equation*}
{\|{\mathit{f}}\|}_{{\mathit{L}}^{2} ((\lambda \times \mu)^{\mathit{n}})}= \Big( {\int_{([0, \mathit{T}] \times {{\mathbb{R}}_{0}})^{\mathit{n}}} ^ {}{{\mathit{f}}^{2} (t_1 , z_1 , ... , t_n , z_n)}\, {d{t_1} \mu(d{z_1}) ... d{t_n} \mu(d{z_n})}} \Big)^{\frac{1}{2}} < \infty,
\end{equation*} 
where $\lambda$ denotes the Lebesgue measure on $[0, \mathit{T}]$. Denote the space of all symmetric functions in ${\mathit{L}}^2 ((\lambda \times \mu) ^n)$ by ${\tilde{L}}^2 ((\lambda \times \mu)^n)$ and for every $\mathit{f} \in {\tilde{L}}^2 ((\lambda \times \mu)^n)$, define
\begin{equation*}
I_{n}(f) := \int_{([0, \mathit{T}]\times {\mathbb{R}}_0)^n}^{}{\mathit{f}(t_1 , z_1 , ... , t_n ,z_n) {\tilde{N}(d{t_1}, d{z_1})} ...  {\tilde{N}(d{t_n}, d{z_n})}},
\end{equation*}
where $\mathit{I_n (f)}$ is the n-fold iterated integral of $\mathit{f}$. For any $\mathit{g} \in {\tilde{L}}^2 ((\lambda \times \mu)^m)$ and $\mathit{f}  \in  {\tilde{L}}^2 ((\lambda \times \mu)^n)$, we have the following relations 
\begin{equation*}
\mathbb{E}\big({\mathit{I}}_m (g) {\mathit{I}}_n (f)\big) = 
\begin{cases}
 0,  {\hspace{28mm} }{n \ne m} \\ (g, f)_{{\mathit{L}}^2 ((\lambda \times \mu)^n)},{\hspace{4mm}}  {n = m} 
\end{cases}
\end{equation*}
where $\mathit{m,n} = 1, 2, ...$ and 
\begin{align*}
(g, f)_{{\mathit{L}}^2 ((\lambda \times \mu)^n)} := \int_{([0, T] \times {\mathbb{R}}_{0})^n}^{} & {g(t_1 , z_1 , ... , t_n , z_n ) f(t_1 , z_1, ... , t_n , z_n)} \notag \\
& d{t_1} \mu (d{z_1}) ... d{t_n} \mu (d{z_n}).\notag
\end{align*}
The following theorem formulates the chaos expansion of some square integrable variables with respect to the Poisson random measure.
\begin{theorem} [\cite{nunno2008malliavin} Theorem 10.2]
Let 
$\mathit{F} \in {{\mathit{L}}^2}(\mathit{P})$ 
be a 
${\mathcal{F}}_T$-measurable random variable. Then 
$\mathit{F}$ 
admits the representation
\begin{equation*}
\sum_{n=0}^{\infty} {\mathit{I}}_{n} ({\mathit{f}}_n),
\end{equation*}
via a unique sequence of elements ${{\mathit{f}}_n \in {\tilde{L}}^2 ((\lambda \times \mu)^n), n = 1, 2, ... .}$ Here we set ${{\mathit{I}}_0}({{\mathit{f}}_0}) := {\mathit{f}}_0$ for the constant values ${\mathit{f}}_0 \in {{\mathbb{R}}_0}$. Moreover, we have 
\begin{equation*}
{{\| {\mathit{F}}\|}^2 _{L^2 (\mathit{P})}} = \sum_{n=0}^{\infty} {n!}{{{\| {\mathit{f}}_n\|}^2 _{L^2 ((\lambda \times \mu)^n)}}} .
\end{equation*}
\end{theorem}
\noindent The Malliavin derivative operator 
$\mathit{D}: {\mathbb{D}}^{1,2} \subset {\mathit{L}}^2 (\mathit{P}) \longrightarrow {{\mathit{L}}^2}({\mathit{P}} \times \lambda \times \mu)$ is indicated as
\begin{equation*}
{\mathit{D} _{t,z}}{\mathit{F}}=\sum_{n=1}^{\infty} {\mathit{n}}{\mathit{I} _{n-1}}{(\mathit{f}_n {(.,{\mathit{t}},{\mathit{z}})})},
\end{equation*}
where $\mathbb{D}^{1,2}$ is the stochastic Sobolev space that consists of all $\mathcal{F}$-measurable random variables $F \in {L}^2 (P)$ with chaos expansion satisfying the convergence criterion
\begin{equation}{\label{1}}
{{\|}{\mathit{F}}{\|}}_{\mathbb{D}^{1,2}}^2 := \sum_{n=1}^{\infty}{ {\mathit{n}}{\mathit{n}!}{{{\|}{{\mathit{f}}_n}{\|}}_{{\mathit{L}}^2 (({\lambda}{\times}{\mu})^n)} ^ 2 } }{<}{\infty}.
\end{equation}

\noindent In fact, according to  (\ref{1}), we have ${{\|}{DF}{\|}}_{L^2 (P \times \lambda \times \mu)}^2={{\|}{F}{\|}}_{\mathbb{D}^{1,2}}^2 $.
We recall some classical properties held for the Malliavin derivatives of random variables. 
\begin{proposition}[\cite{nunno2008malliavin} Theorem 12.7]\label{pro22}
Let 
$ F,G \in {\mathbb{D}^{1,2}}.$
Then 
$ FG \in {\mathbb{D}^{1,2} } $
 and 
\begin{equation*}
D_{t,z} (FG) = F{D_{t,z} G}+ G{D_{t,z} F} + {D_{t,z} F} {D_{t,z} G}.
\end{equation*}
\end{proposition}
\begin{proposition}[\cite{nunno2008malliavin} Theorem 12.8]\label{pro23}
Let 
$ F \in {\mathbb{D}^{1,2}}$ 
and let 
$\varphi$
be a real continuous function on 
$\mathbb{R}$. Suppose 
$\varphi (F) \in {L^2}(P)$
 and 
$\varphi (F+{D_{t,z} F}) \in {L^2 (P \times \lambda \times \mu)}.$
Then 
$\varphi (F) \in {\mathbb{D}^{1,2}}$
and 
\begin{equation*}
D_{t,z}(\varphi (F))=\varphi (F + {D_{t,z}F})- \varphi(F).
\end{equation*}
\end{proposition}
\noindent Let $\mathit{X} = \mathit{X(t, z)}, 0 \le {\mathit{t}} \le {\mathit{T}}, \mathit{z} \in {{\mathbb{R}}_0},$ be a stochastic process such that $\mathit{X}(\mathit{t}, \mathit{z})$ is an $\mathcal{F}_{\mathit{T}}$-measurable random variable and has a chaos expansion with the elements ${{\tilde{f}}_n \in {\tilde{L}}^2 ((\lambda \times \mu)^n), n = 1, 2, ... .}$ . Suppose that,
\begin{equation*}
\sum_{n=0}^{\infty} {(n+1)! {\| {{\tilde{f}}_n} \|}^2 _{{\tilde{L}}^2 ((\lambda \times \mu)^{(n+1)})}} < {\infty},
\end{equation*}
then the Skorohod integral $\delta (X)$ with respect to $\tilde{N}$ is defined as 
\begin{equation*}
\delta (\mathit{X}) = \int_{0}^{T}{\int_{{\mathbb{R}}_0}^{}{X(t, z) \tilde{N} (dz, \delta t)}} := \sum_{n=0}^{\infty} {{I}_{n+1} ({\tilde{f}}_n)}.
\end{equation*}
\noindent  Denote by $\mathbb{L} ^{1,2}$ the set of processes $g$ in $L^2 (P \times \mu)$ such that $g(z)$ belongs to $\mathbb{D}^{1,2}$ for almost all $z$ and there exists a measurable version of the two-parameter process $Dg$ such that
\begin{equation*}
\mathbb{E} \Big( {\|{Dg}\|}^2 _{L^2 {(\lambda \times \mu)^2}} \Big) < \infty.
\end{equation*}
\begin{lemma}[\cite{nualart2018introduction} Lemma 10.2.4]
We have that  $\mathbb{L} ^{1,2} \subset Dom{\hspace{0.2mm}} \delta$.
\end{lemma}
\noindent Nualart \cite{nualart2018introduction} in proposition 10.2.2 has proved that $\delta$ is the adjoint operator of $D$. The duality formula between the Malliavin derivative and the Skorohod integral in the setting of the Poisson process is stated in the following proposition.
\begin{proposition}[\cite{nunno2008malliavin} Theorem 12.10]
Let 
$X(t,z), t \in {\mathbb{R}_{+}}, z \in {\mathbb{R}_0}, $
be Skorohod integrable and 
$F \in {\mathbb{D}^{1,2}}. $
Then 
\begin{equation*}
{\mathbb{E}}{\Big (}{\int_{0}^{\infty} {\int_{\mathbb{R}_0}^{} {X(t,z) {D_{t,z}F}}\, {\mu (dz)}}\, dt}{\Big )}={\mathbb{E}}{\Big (}{F\int_{0}^{\infty} {\int_{\mathbb{R}_0}^{} {X(t,z)}\, }\, {\tilde{N}}(dz, \delta t)}{\Big )}.
\end{equation*}
\end{proposition}
\noindent The following result presents the fundamental theorem of calculus for Poisson random measure.
\begin{proposition}[\cite{nunno2008malliavin} Theorem 12.15]\label{derjump}
Let 
$X=X(s,y), (s,y) \in [0,T] \times {\mathbb{R}_{0}}, $
be a stochastic process such that 
\begin{equation*}
\mathbb{E}{\Big (}{\int_{0}^{\infty}} {{\int_{\mathbb{R}_{0}}^{} {X^2 (s,y)}\, \mu (dz)}\, ds}{\Big)}{<}{\infty}.
\end{equation*}
Assume that 
$X(s,y) \in {\mathbb{D}^{1,2}}$
 for all 
$(s,y) \in [0,T] \times {\mathbb{R}_{0}}, $
and
$D_{t,z}X(.,.)$
 is Skorohod integrable with 
\begin{equation*}
\mathbb{E}{\Big ( }{\int_{0}^{T} {\int_{\mathbb{R}_{0}}^{} \Big{(}\int_{0}^{T} {\int_{\mathbb{R}_{0}}^{} {D_{t,z}X(s,y)}\, }\, \tilde{N}(dy, \delta s) \Big{)}^2\, \mu (dz)}\, dt} {\Big ) } {<}{\infty}.
\end{equation*}
Then for every real function $\gamma$, 
\begin{equation*}
\int_{0}^{t} {\int_{\mathbb{R}_{0}}^{} {\gamma ({X(s,y)})}\, }\, \tilde{N}(dy, \delta s) \in {\mathbb{D}^{1,2}},
\end{equation*}
and
\begin{align*}
D_{t,z}\int_{0}^{T} {\gamma (X(s, y))}& {\tilde{N}}(dy, \delta s) ={\gamma (X(t, z))} +   \notag\\
&\int_{t}^{T} \int_{{\mathbb{R}}_{0}}^{} \Big{(}{\gamma \big{(}X(s, y) + D_{t,z} X(s, y)\big{)}} - {\gamma (X(s, y))}\Big{)} {\tilde{N}}(dy, \delta s) .
\end{align*}
\end{proposition}
\section{Mean-Field SDEs with jump}
\noindent
 Consider the Wiener space  
${(\Omega_1,\mathcal{F}^W,{P^W})}$
which the filtration 
${\mathcal{F}^W}$ is generated by 
${W:=\{(W^{1}(t), W^{2} (t), ... , W^{d} (t))^{'}{\mid} t\in[0,T]\}},$ 
where 
${W^{i} , i=1,2, ... ,d}$
 are independent with each other, $\mathit{T\in{\mathbb{R}^+}}$ and 
${C^{'}}$
 denotes the transpose of any vector or matrix 
${C}$. Also, suppose that 
${\mathit{N}(dz,dt)}$
is a Poisson random measure on product measurable space 
$([0,T]\times {\mathbb{R}_{0}}, {\mathcal{B} ([0,T])}\times {\mathcal{B}(\mathbb{R}_{0})})$, 
and $\tilde{N}(dz, dt)$ denotes compensated Poisson random measure by the Poisson space ${(\Omega_2,\mathcal{F}^N,{P^N})}$.
\noindent We assume that the Brownian motion and the Poisson random measure are stochastically independent under 
${P}=P^W\times P^N$. In the sequel, we call the Wiener-Poisson space ${(\Omega,\mathcal{F},{P})}$  where $\Omega=\Omega_1 \times \Omega_2$ and ${\mathbb{F}:=\{ \mathcal{F}(t) \mid t \in [0,T]\}}$ is the right-continuous and
${{P}}$-complete natural filtration generated by both the Brownian motion and the Poisson random measure. 
 Now consider the following mean-field SDE with jumps
\begin{align}\label{sde1}
dX_t ^x&=b(t,\rho_{X_t ^x})X_t^xdt+\sum_{n=1}^d{{\sigma_n}(t,X_t^{x},\pi_{X_t ^x})d{W_t ^n}}+\int_{\mathbb{R}_0}^{} {\lambda(t,X_t^x, z, \eta_{X_t ^x})\tilde{N}(dz,dt)}, \nonumber\\
X(0)&=x \in {\mathbb{R}}^k-\{0\} , 
\end{align}
\noindent where
the functions $b:  [0,T] \times {\mathbb{R}^k}  \rightarrow {\mathbb{R}^k}$,
${\sigma_n}:\Omega \times [0,T] \times {\mathbb{R}^k} \times {\mathbb{R}^k} \rightarrow {\mathbb{R}^k}$ and 
$\lambda :\Omega \times [0,T] \times {\mathbb{R}^k} \times {\mathbb{R}_0} \times {\mathbb{R}^k} \rightarrow {\mathbb{R}^{k \times l}}$ are $\mathbb{F}$-measurable, 
\begin{center}
${\rho_{X_t ^x}:={\mathbb{E}(X_t^x)}}, \hspace{5mm} {\pi_{X_t ^x}:={\mathbb{E}(\psi(X_t^{x}))}}, \hspace{5mm} {\eta_{X_t ^x}:={\mathbb{E}(\xi(X_t^{x}))}}$,
\end{center}
 in which
$\psi$, $\xi : {\mathbb{R}^k \to {\mathbb{R}^k}}$ are also measurable functions. We assume the following conditions for these functions throughout the paper and denote $|y|$ for the Euclidean norm in the associated space of $y$. \\
%
\begin{assumption}\label{assum1} 
The functions $b, \sigma_n, \lambda, \psi, \xi $ satisfy the following conditions. 
\begin{enumerate}
\item The functions $\sigma_n, \lambda, \psi, \xi $ are continuously differentiable with the Lipschitz bound parameter $\mathit{K}$, i.e., there exists a positive constant $K$ and a function ${\varrho :{\mathbb{R}_0}\rightarrow \mathbb{R}}$ satisfying $\int_{\mathbb{R}_0}^{}{{\varrho ^p}(z) \mu(dz)} < \infty$, for every $p \geq 1$, such that for every $t, s \in [0, T]$ and stochastic processes $X$ and $Z$ where, $n=1, ...,d$, we have
\begin{align}
\big|{\sigma_{n}} (t, X_t ^x, {\pi}_{X_t ^x})&- {\sigma_{n}} (s, Z_s ^z, {\pi}_{Z_s ^z})\big|^2 \nonumber \\
& \leq K \big( {|t-s|^2 + | X_t ^x - Z_s ^z |^2 + \mathbb{E}\vert X_t ^x - Z_s ^z\vert^2 }\big), \notag\\
\big|{\lambda} (t, X_t ^x, z, {\eta}_{X_t ^x})& -{\lambda} (s, Z_s ^z, z, {\eta}_{Z_s ^z})\big|^2 \nonumber \\
&\leq {\varrho(z)} \big( {|t-s|^2 +| X_t ^x - Z_s ^z |^2 + \mathbb{E}\vert X_t ^x - Z_s ^z\vert^2 }\big).\label{condition}
\end{align}
\item The functions $b, \sigma_n, \lambda, \psi, \xi$ has a linear growth.
\end{enumerate}
\end{assumption}
\noindent We assume that the function $f(.):[0,T] \rightarrow \mathbb{R}^k$ is the unique bounded non-zero solution to the following differential equation 
\begin{equation}\label{ft}
-b(t, f(t))f(t) + \dot{f(t)}=0, \qquad f(0)=x.
\end{equation}
 \begin{remark}
 We observe that this condition is not an unusual limitation with the following example when $b(t, f(t))$  has a linear form. Let $b(t, f(t)) = c + af(t)$ with the condition $\sup_{0\le t \le T}Kae^{ct} >1$ when  $K=\frac{f(0)}{c+af(0)}$. Solving the equation \eqref{ft} results in the bounded function 
 \begin{align*}
f(t) = \frac{Kce^{ct}}{1-Kae^{ct}}.
\end{align*}
\end{remark}
 \noindent Consider the following stochastic differential equation
\begin{equation}\label{sde2}
d{S_t}={ \sum_{n=1}^{d}\{{\tilde{\sigma_n} (t, S_t, {\pi}_{ S_t})}}\}d{W_t ^n} + \int_{\mathbb{R}_0}^{}\{{\tilde{\lambda} (t, S_t, z, {\eta}_{S_t})}\} \tilde{N}(dz, dt), \quad S_0=x,
\end{equation}
where 
\begin{align}
\tilde{\sigma_{n}} (t, S_t, {\pi}_{S_t})&=xf^{-1}(t){\sigma}_n (t, \frac{1}{x}f(t)S_t, {\pi}_{\frac{1}{x}f(t)S_t}), \hspace{5mm} n=1,...,d, \notag\\
\tilde{\lambda} (t, S_t, z, \eta_{S_t}) & =xf^{-1}(t)\lambda(t,  \frac{1}{x}f(t)S_t ,z, \eta_{\frac{1}{x}f(t)S_t})\notag\\
&=: {e_{t, X_t ^x}}\lambda(t, \frac{1}{x}f(t)S_t ,z, \eta_{\frac{1}{x}f(t)S_t})\label{tilde}
\end{align}
which indicates that $f^{-1}$ serves as the inverse of $f$. Obviously, under Assumption \rm{\ref{assum1}}, the functions $\tilde{\sigma_{n}}$ and $\tilde{\lambda}$ are Lipschitz and satisfy again in \eqref{condition}. According to the Itô formula, it is shown that the SDE (\ref{sde1}) has a unique solution if and only if the SDE (\ref{sde2}) has a unique solution as we will see in the following theorem. 
\begin{theorem}\label{XT}
Under Assumption \rm{\ref{assum1}}, the SDE (\ref{sde1}) has a unique solution $X_t ^x \in L^p(P)$, $p \geq 1$.  In fact, for every solution $S_t \in L^{p}(P)$ of SDE (\ref{sde2}), the process $\frac{1}{x}f(t) S_t$ is a solution of SDE (\ref{sde1}), and vice versa.
\end{theorem}
\begin{proof}
\noindent We will prove the assertion in 3 steps.\\
Step 1. Suppose that $X_t ^x$ is a solution to the SDE (\ref{sde1}). By using the Itô formula and Equation (\ref{ft}) , it is obvious that $exp\{-{\int_{0}^{t}{b(s, f(s))ds}}\}X_t ^x= xf^{-1}(t)X_t ^x$ is a solution to the SDE (\ref{sde2}). On the other hand,  assume that $S_t$ is a solution to the SDE (\ref{sde2}) and define ${\tilde{X}_t ^{\tilde{x}} := exp\{{\int_{0}^{t}b(s, f(s))ds}}\}S_t$ with $\tilde{x}=x$. Itô formula deduces that  
\begin{align*}
d{\tilde{X}_t ^{\tilde x}}&=exp\{\int_{0}^{t}b(s, f(s))ds\} d{S_t} + b(t, f(t))exp\{{\int_{0}^{t}{b(s, f(s))}ds}\}{S_t}dt\notag\\
&= b(t,f(t))\tilde{X}_t^{\tilde{x}}dt+\sum_{n=1}^d{\{\sigma_n(t, \tilde{X}_t^{\tilde x}, {\pi}_{\tilde{X}_t^{\tilde x}})\}d{W_t ^n}}\notag\\ 
&+\int_{\mathbb{R}_0}^{} \{ \lambda (t, \tilde{X}_t^{\tilde{x}}, z, \eta_{\tilde{X}_t ^{\tilde x}})\}\tilde{N}(dz,dt), 
\end{align*}
where we used Equation (\ref{ft}) in the last equality. Also, taking expectation from above equality and using the martingale property of two ends integral to derive 
$$\mathbb{E}(X_t ^x)= \int_0^t b(s,f(s))\mathbb{E}(X_s ^x)ds.$$  
Since $\mathbb{E}(X_0)=x=f(0)$ and Equation (\ref{ft}) has a unique solution, we deduce that $\mathbb{E}(X_t ^x)=f(t)$, for every $0 \leq t \leq T$. Therefore, $\tilde{X}_t^{\tilde x}$ satisfies the SDE (\ref{sde1}). \\  
Step 2. According to Assumption \ref{assum1}, the existence of the solution of the SDE (\ref{sde2}) and the boundedness of its $p$-moments can be easily obtained by the Gronwall inequality in the same proof of Theorem 3.1 in \cite{song2022regularity} or Theorem 3.1 in \cite{hao2016mean}, and we omit this proof here. Then the SDE (\ref{sde1}) has a solution and also due to linear growth of $b$ and boundedness of $f$,
\begin{equation*}
\mathbb{E}(X_t^x)^p  = exp\{{\int_{0}^{t}{pb(s, f(s))ds}}\}\mathbb{E}(S_t^p) < \infty.
\end{equation*} 
Step 3. Now we give the proof of uniqueness for the SDE (\ref{sde1}). Consider the two solutions ${X_t}$ and $\hat{X}_t$ to the SDE (\ref{sde1}) with initial conditions $X_0 = x$ and $\hat{X}_0 = \hat{x}$ respectively. We have,
\begin{align*}
\mathbb{E}{\big(}{|X_t ^x - {\hat X}_t ^{\hat x} |}^p{\big)} &= \mathbb{E} \big( {|e_{t, X_t ^x} ^{-1}{S}_t -  e_{t, {\hat{X}_t ^{\hat x}}} ^{-1} \hat{S}_t|^p}\big) \notag\\
&\le \mathbb{E}\big( |e_{t, X_t ^x} ^{-1}{S}_t - e_{t, X_t ^x} ^{-1}{\hat S}_t|^p\big) + \mathbb{E}\big(| e_{t, {X_t ^ x}} ^{-1} \hat{S}_t -e_{t, {\hat{X}_t ^{\hat x}}} ^{-1} \hat{S}_t|^p \big) \notag\\
&\le {|e_{t, X_t ^x}^{-1}|^p}{\mathbb{E} \big({|S_t - \hat{S}_t|^p}\big)} + {|{{e_{t, {X_t ^x}}^{-1}} -  e_{t, {\hat{X}_t ^{\hat x}}}^{-1}}|^p}{\mathbb{E}{\big( |\hat{S}_t|^p \big)}}.
\end{align*}
 Uniqueness of $f$ when $x=\hat x$, and the uniqueness of the solution of (\ref{sde2}) complete the proof.
\end{proof}
 \noindent The following lemma, which will be used frequently in the sequence, can be proved with the same technique of the proof of  Theorem 3.2 in \cite{kunita2004stochastic}, for the solution of SDE (\ref{sde2}). So we omit the proof.
\begin{lemma}\label{t1}
Set ${\tilde{\sigma}} = ({\tilde{\sigma}}_1 , ..., {\tilde{\sigma}}_d)$, then for any $p\ge 2$, there exists a positive constant $C_p$ such that 
\begin{align*}
\mathbb{E}\big(\sup_{0<s\le t} |S(s)|^p \big)\le C_p \Big \{ |x|^p &+ \mathbb{E}\Big( \Big( \int_{0}^{t}{|\tilde{\sigma} (s, S_s , \pi_{S_s })|^2}ds \Big)^{p/2}\Big)\notag\\
& + \mathbb{E}\Big( \Big( \int_{0}^{t}{\int_{\mathbb{R}_0}^{}{|\tilde{\lambda}(s, S_s , z, \eta_{S_s})|^2 \mu(dz)ds}} \Big) ^{p/2} \Big) \notag\\
&+ \mathbb{E} \Big( \int_{0}^{t}{\int_{\mathbb{R}_0}^{}{|\tilde{\lambda}(s, S_s, z, \eta_{S_s})|^p}\mu(dz)ds}\Big) \Big\}.
\end{align*}
\end{lemma}
\noindent and thus from the linear growth of $\tilde{\sigma}$, $\tilde{\lambda}$, $b(t, \rho_{X_t ^x})$ and the boundedness of $f$, for some constant $C$, we conclude that
\begin{equation}\label{equasup}
\mathbb{E}\big( {\sup_{0 \le s \le t}{|X_s ^x|^p}}\big) \le C \mathbb{E}\big({\sup_{0\le s \le t}{|S_s ^x|^p}}\big) < \infty.
\end{equation}

\section{Stochastic flow and weak derivative of the solution}
\noindent In this section, we prove the existence of the stochastic flow of the solution of (\ref{sde1}) as well as its Malliavin differentiability. Then we find a new relation between them to be considered in driving an explicit formula of sensitivity analysis in the next section.

\subsection{Existence of stochastic flow and its expression}
\noindent Here some more assumptions regarding the coefficients need to be made. 
\begin{assumption}\label{assume2}
 Denote ${\partial}_i g= \frac{\partial}{{\partial}{x_i}} g,$ for any $\mathit{g = g(t, x_1, x_2, ..., x_n)}$ with $x_i \in \mathbb{R}^k$, $t \in [0,T]$ and $i= 1, 2, ..., n.$ 
\begin{enumerate}
\item The functions $b(.,.)$ and ${{\sigma}_n}(.,.,.), n=1, ..., d,$ are differentiable with,
\begin{align}\label{eq1}
\|\partial _1 b\|_{\infty} &:= \sup_{y \in \mathbb{R}^k} |\partial _1 b(t, y)| < \infty,\notag\\
\|\partial _i \sigma_n\|_{\infty} &:= \sup_{y,z \in \mathbb{R}^k} |\partial _i \sigma _n(t,y,z)| < \infty, \quad i=1,2, \notag\\
\inf_{x \in \mathbb{R}^k, z \in \mathbb{R}_0}& |det|\mathbf{1}+\partial_1 \lambda(s, x, z, \eta_x)||> 0,
\end{align}
where $det|.|$ denotes the determinant of the matrix and also for every $x,z \in \mathbb{R}^k$, there is a constant $C$ such that, $i=1,2$,
\begin{align}\label{eq3}
|\partial_1 b(t, \rho _{X_t ^x}) - \partial_1 b(t, \rho _{Z_s ^z})| &\le C \Big({|t - s| + \kappa \mathbb{E}|X_t ^x - Z_s ^z|}\Big), \notag\\
 \notag\\
|\partial_i \sigma _n(t, X_t ^x, \pi _{X_t ^x}) - \partial_i \sigma _n(&t, Z_t ^z, \pi _{Z_s ^z})| \notag\\ &\le C\Big({|t - s| + |X_t ^x - Z_s ^z|  + \kappa \mathbb{E}|X_t ^x - Z_s ^z|}\Big).
\end{align}
\item For every $t, s \in [0, T]$ and stochastic processes $X_t ^x$ and $Z_s ^z$ belong to $L^2({P})$ and  $i=1,3$ we have,
\begin{equation}\label{eq2}
\|\partial _i \lambda\|_{\infty} := \sup_{X_t ^x \in \mathbb{R}^k, z \in \mathbb{R}_0} |\partial _i \lambda(t, X_t ^x, z, \eta _{X_t ^x})| < \infty, 
\end{equation}
\begin{align}\label{eq4}
|\partial_i \lambda(t, X_t ^x, z, \eta _{X_t ^x}) &- \partial_i \lambda(t, Z_s ^z, z, \eta _{Z_s ^z})|\notag\\
&\le \varrho(z) C\big({|t - s| + |X_t ^x - Z_s ^z|  + \kappa \mathbb{E}|X_t ^x - Z_s ^z|}\big).
\end{align}
\end{enumerate}
\end{assumption}
\noindent  Recall that for given $\mathit{x} \in \mathit{L}^2 (P)$, the stochastic flow of $\mathit{X}_t ^x$ is defined as,
\begin{equation*}
\frac{\partial}{\partial {x}} X_t ^x := L ^2 - \lim_{\epsilon \to 0} \frac{1}{\epsilon} (X_t ^{x+{\epsilon}x} - X_t ^x).
\end{equation*}
Under the above assumptions, we will show that the flow of the process $X_t$ exists. To do this, we notice that regarding Theorem \ref{XT} we have
\begin{equation}\label{eq16}
\frac{\partial}{\partial x}X_t ^x = \partial_1b(t, \rho _{X_t ^x}) \frac{\partial}{\partial x} \rho _{X_t ^x} e^{{\int_{0}^{t}{b(s, \rho _{X_s ^x})}ds}} S_t ^x + e^{{\int_{0}^{t}{b(s, \rho_{X_s ^x})}ds}}{\frac{\partial}{\partial x}}S_t ^x,
\end{equation}
where $X_0 = S_0 = x$. Therefore, due to the uniqueness of solution  (\ref{sde2}), $b$ has linear growth, $f$ is bounded, and the fact that $\partial_1 b(t, \rho_{X_t ^x})$ is bounded, it is sufficient to demonstrate $\frac{\partial}{\partial x}S_t ^x$ exists and satisfies the following equation,
\begin{align}\label{sde5}
\frac{\partial}{\partial{x}}{S_t ^x} = 1 &+ \sum_{n=1}^{d}{\int_{0}^{t}{\Big( {{\tilde B_s ^{x,n}}{\frac{\partial}{\partial x}S_s ^x} + \tilde \beta _s ^{x,n}}\Big)d{W_s ^n}}}\notag\\
&+ \int_{0}^{t}{\int_{\mathbb{R}_0}^{}{\Big( {{\tilde M_{s,z} ^x}\frac{\partial}{\partial x}S_s ^x + \tilde \gamma _s ^x}\Big)\tilde{N}(dz, ds)}},
\end{align}
where
\begin{align}
&\hspace{10mm}{\tilde B_t ^{x,n}}:=\partial_1{\tilde \sigma_n}(t, S_t ^x, \pi_{S_t ^x}) ,
\hspace{18mm}\tilde \beta _t ^{x,n} := \partial_2{\tilde \sigma_n}(t, S_t ^x, \pi_{S_t ^x})\frac{\partial}{\partial x}\pi_{S_t ^x}, \notag\\
&\hspace{10mm}{\tilde M_{t,z} ^x}:=\partial_1{\tilde \lambda}(t, S_t ^x, z, \eta_{S_t ^x}) ,
\hspace{18mm}\tilde \gamma _t ^x := \partial_3{\tilde \lambda}(t, S_t ^x, z, \eta_{S_t ^x})\frac{\partial}{\partial x}\eta_{S_t ^x} .\notag
\end{align}
Then by the Itô formula, we obtain the SDE whose solution is the flow of $X_t$. Therefore, we will show in the following two propositions that SDE (\ref{sde5}) has a unique solution whose solution is the stochastic flow of $S_t$. A detailed description of these proofs, that are similar to the proof of Theorem 1.1 in \cite{song2022regularity}, can be found in \ref{A1}.\\
\begin{proposition}\label{flowS}
Under Assumption \rm{\ref{assume2}}, for each $t \in [0, T]$  the stochastic flow of $\mathit{S}_t ^x$ exists. 
\end{proposition}
\begin{proposition}\label{flowS2}
Under Assumption \rm{\ref{assume2}}, for each $t \in [0, T]$  the unique solution of SDE (\ref{sde5}) is the stochastic flow of $\mathit{S}_t ^x$. 
\end{proposition}
\noindent  To succeed, we first note that according to (\ref{tilde}), the following relations satisfy.  
\begin{align}
{B_t ^{x,n}}& :=\partial_1{\sigma_n}(t, X_t ^x, \pi_{X_t ^x})=\tilde{B}_t ^{x,n} ,
  \hspace{10mm} \beta _t ^{x,n} := \partial_2{\sigma_n}(t, X_t ^x, \pi_{X_t ^x})\frac{\partial}{\partial x}\pi_{X_t ^x}=\tilde{\beta} _t ^{x,n}, \notag\\
{M_{t,z} ^x}&:=\partial_1{\lambda}(t, X_t ^x, z, \eta_{X_t ^x})=\tilde{M}_{t,z} ^x ,
\hspace{8mm}\gamma _t ^x := \partial_3{\lambda}(t, X_t ^x, z, \eta_{X_t ^x})\frac{\partial}{\partial x}\eta_{X_t ^x}=\tilde{\gamma}_t ^x .\notag
\end{align}
\begin{theorem}\label{flow}
Under Assumption \rm{\ref{assume2}}, for each $t \in [0, T]$ the stochastic flow of $\mathit{X}_t ^x$ exists and satisfies the following equation,
\begin{align}\label{sde3}
\frac{\partial}{\partial{x}}{X_t ^x} = 1 &+ \int_{0}^{t}{\Big( {{A_s ^x}\frac{\partial}{\partial x}X_s ^x + \alpha _s ^x}\Big)ds} + \sum_{n=1}^{d}{\int_{0}^{t}{\Big( {{B_s ^{x,n}}\frac{\partial}{\partial x}X_s ^x + \beta _s ^{x,n}}\Big)d{W_s ^n}}}\notag\\
&+ \int_{0}^{t}{\int_{\mathbb{R}_0}^{}{\Big( {{M_{s,z} ^x}\frac{\partial}{\partial x}X_s ^x + \gamma _s ^x}\Big)\tilde{N}(dz, ds)}},
\end{align}
where 
\begin{align}
{A_s ^x}:=b(s, \rho_{X_s ^x}), \qquad \alpha _s ^x := \partial_1 b(s, \rho_{X_s ^x}){X_s ^x}\frac{\partial}{\partial x}\rho_{X_s^x}. \notag
\end{align}
\end{theorem}
\begin{proof}
It is proven by the Itô formula and (\ref{eq16}). 
\end{proof}
\noindent In the sequel, we state some expressions for the stochastic flow of ${X_t}$. Let $Y_t ^x$
is the solution of the homogeneous SDE (\ref{sde3}), $\mathit{i.e.}$,\\
\begin{equation}\label{Y}
d{Y_t ^x}={A_t ^x}{Y_t ^x} dt + {\sum_{n=1}^d B_t ^{x,n}}{Y_t ^x}d{W_t ^{n}} + \int_{{\mathbb{R}_{0}}}^{} {M_{t,z} ^ x}{Y_t ^x}{\tilde{N}}\, (dz, dt).
\end{equation}
Itô formula deduce that ${\mathit{Y_t ^x}} =\mathit{ e} ^{\mathit{H}(\mathit{t})}$, where the process $H(t)$ satisfy 
\begin{align*}
&H(t)=\int_{0}^{t}\big ( {A_s ^x - \frac{1}{2} {\sum_{n=1}^d B _s ^{x,n}}} + \int_{\mathbb{R}_{0}}^{}( ln(\mathbf{1}+{M_{s,z} ^x})-{M_{s, z} ^x})\, {\mu}(dz)\big)\, ds \notag\\
&{\hspace{9mm}}+ \sum_{n=1}^d {\int_{0}^{t} { B_s ^{x,n}}\, dW_s ^{n}}+{\int_{0}^{t} {\int_{\mathbb{R}_{0}}^{} {ln(\mathbf{1}+ {M_{s,z} ^x})}\, }\, {\tilde{N}(dz,ds)}}.
\end{align*}
According to (\ref{eq1}), (\ref{eq2}) and Gronwall inequality it is common to show that for every $p \ge 1$, the process $Y_t ^x $ belongs to $L^p$. In fact, there exists a constant $C_p$ such that,
\begin{equation}\label{ytp}
\mathbb{E}\Big(\sup_{t \in [0, T]} |Y_t ^x|^p \Big) \le C_p.
\end{equation}
For more details see \cite{menaldi2008stochastic}. Also, from Assumption \ref{assum1}, for each $T>0$ and $p \ge 2$ there exists $R= R(T, p)$ such that 
\begin{equation}\label{eq15}
|{A_t ^x}| + |{B_t ^{x,n}}| + {\int_{\mathbb{R} _0}^{}{{|{M_{t,z} ^x}|}}^p} \mu (dz) \le R,
\end{equation}
and
\begin{equation}\label{eq11}
\mathbb{E}\Big{\{} \Big({{\int_{0}^{T}{|\alpha _t ^x |}ds}}\Big) ^p + \sum_{n=1}^d {\Big( {\int_{0}^{T}{{{|\beta _t ^{x,n}|}^2 ds}}  \Big)^{p/2}} + \Big( {\int_{0}^{T}{\int_{\mathbb{R} _0}^{}{{|\gamma _t ^x|}^2 \mu(dz)ds}}}}\Big)^{p/2} \Big{\}} \le R.
\end{equation}
\noindent Therefore the conditions of Lemma 5.19 in \cite{menaldi2008stochastic} hold and it results that for every $\mathit{t} \in [0, T]$, $\mathit{Y_t ^x}$ is $\mathit{P}$-a.s. invertible and
\begin{equation}\label{equa1}
{{det {\hspace{0.75mm}}   { {(Y_t ^{x})}^{-1}}} \in {\bigcap_{p \ge 1}^{}{{L^p}(\Omega)}}}.
\end{equation}
and also
\begin{align*}
&(Y_t ^x)^{-1} = \mathbf{1} - {\int_{0}^{t} (Y_t ^x)^{-1} \tilde{A_s} ds} - \sum_{n=1}^{d}{\int_{0}^{t}{(Y_s ^x)^{-1} B_s ^{x,n}}d W_s ^{n}} \notag\\
&{\hspace{17mm}} - {\int_{0}^{t}{\int_{{\mathbb{R}}_0}^{}{(Y_s ^x)^{-1}} M_{s,z} ^x ( \mathbf{1} + M_{s,z} ^x)^{-1}}\tilde{N}(dz,ds)},
\end{align*}
where
\begin{equation*}
\tilde{A_t}:= A_t ^x -{\sum_{n=1}^{d}(B_t ^{x,n})^2}- {\int_{{\mathbb{R}}_0}^{} (M_{t,z} ^x)^2 (\mathbf{1} + M_{t,z} ^x)^{-1} \mu(dz)}.
\end{equation*}
\noindent Now, using the technique inspired of \cite{banos2018bismut} we can get an expression for $\frac{\partial}{\partial{x}}{X_t ^x}$ of a mean-field SDE with jumps.
\begin{lemma}\label{lem1}
For every $t \in [0,T]$, there exists a stochastic process  $u(t)$ such that 
\begin{equation}\label{eq7}
 {\frac{\partial}{\partial{x}}{X_t ^x}}={Y_t ^x}{u(t)}. \\
\end{equation}
\end{lemma}
\begin{proof}
\noindent Let us define the process $u(t)$ as the solution of the following SDE
\begin{align}\label{eq8}
du(t)={A_t ^*} {(Y_t ^x)}^{-1} dt + \sum_{n=1}^{d} \beta_t^{x,n}{( Y_t ^x)}^{-1} d{W_t ^n} + \int_{{\mathbb{R}}_{0}}^{}  M_{t,z} ^* {(Y_t ^x)}^{-1}\,{\tilde{N}}(dz,dt),
\end{align}
where
\begin{align*}
&A_t ^*={\alpha}_t ^{x} -\sum_{n=1}^d{\beta}_t ^{x,n} {B}_t ^{x,n} - {\int_{{\mathbb{R}}_{0}}^{} {\frac{{M_{t,z} ^{x}}{{\gamma}_t ^{x}}}{\mathbf{1}+ {M_{t,z} ^{x}}}}\,{\mu(dz)}}, \qquad M_{t,z}^*={\frac{{\gamma}_t ^{x}}{\mathbf{1}+{M_{t,z} ^{x}}}}.
\end{align*}
Applying the Itô formula to derive
\begin{align*}
d(Y_t ^{x} u_t)&={A_t ^{x}}{Y_t ^{x}} u_t dt + \sum_{n=1}^{d} {B_t ^{x,n}} {Y_t ^{x}}  u_t d{W_t ^n} + {\int_{{\mathbb{R}}_{0}}^{} {M_{t,z} ^{x}}{Y_t ^{x}} u_t\, {\tilde{N}}(dz,dt)} \notag\\
&+{A_ t ^{*}}dt + \sum_{n=1}^d {\beta_t ^{x,n}} d{W_t ^n}+{\int_{{\mathbb{R}}_{0}}^{} {M_{t,z}^{*}}{{\tilde{N}}(dz,dt)}}\,  + \sum_{n=1}^d {B_t ^{x,n}}\beta_t ^{x,n}dt \notag\\
&+{\int_{{\mathbb{R}}_{0}}^{} {M_{t,z} ^{x}}{M_{t,z} ^{*}}\, {{\tilde{N}}(dz,dt)}}+{\int_{{\mathbb{R}}_{0}}^{} {M_{t,z} ^{x}}{M_{t,z} ^{*}}\, {\mu (dz)dt}} \notag\\
&=\Big({A_t ^x}Y_t ^{x} u_t+ \alpha _t^x\Big)dt+ \sum_{n=1}^{d}\Big({{B_t ^{x,n}}Y_t ^{x} u_t+ \beta _t ^{x,n}}\Big)d{W_t ^n}\notag\\
&+ {\int_{\mathbb{R}_0}^{}{\Big({{M_{t,z} ^x}Y_t ^{x} u_t+ \gamma _t ^x}\Big)\tilde{N}(dz, ds)}}.
\end{align*}
Now, the uniqueness property results in that 
$Y_t ^{x} u(t)= {\frac{{\partial}}{\partial{x}}}{X_t ^{x}}$.\\
\end{proof}
\begin{corollary}\label{coro}
The stochastic flow $\frac{{\partial}}{\partial{x}}X_t ^{x}$ of the SDE ({\ref{sde3}}) belongs uniformlly to $L^p(P)$  for any $p \ge 1$, i.e., 
$$\mathbb{E} \Big(\sup_{0 \leq t \leq T} \vert \frac{{\partial}}{\partial{x}}X_t ^{x} \vert^p \Big)< \infty.$$
Indeed, from the same technique of the proof of Theorem 3.2 in \cite{kunita2004stochastic}, in connection to \eqref{eq15}, \eqref{eq11}, and \eqref{eq8} we derive that for any $p \ge 1$, 
\begin{align*}
&\mathbb{E} {\Big(\sup_{t \in [0, T]} |u(t)|^p \Big)}\\
&  \le C_p\Bigg\{ |u_0|^p + \mathbb{E}\Big({ \big({\int_{0}^{T}{| {A^*}(Y_s ^x)^{-1}|} ds}  \big)^{p}} \Big)+ \mathbb{E}\Big({ \big({\int_{0}^{T}{|\sum_{n=1}^{d}{{\beta}_t ^{x,n}}(Y_s ^x)^{-1}|} ds}  \big)^{p}} \Big) \notag\\ 
 &+ \mathbb{E}\Big({ \big({\int_{0}^{T}{\int_{\mathbb{R}_0}^{}{|M_{s,z} ^* (Y_s ^x)^{-1}|}^2}\mu(dz) ds}  \big)^{p/2}}\Big) \notag\\ 
& + \mathbb{E}\Big({{\int_{0}^{T}{\int_{\mathbb{R}_0}^{}{|M_{s,z} ^* (Y_s ^x)^{-1}|}^p}\mu(dz) ds}}\Big) \Bigg\},\notag
\end{align*}
\noindent by utilizing equations \eqref{eq1}, \eqref{eq11}, and \eqref{equa1}, we can conclude that the solution of the stochastic differential equation (SDE) given in \eqref{sde3} belongs to $L^p(P)$  for any $p \ge 1$ and 
\begin{align*}
\mathbb{E} {\Big(\sup_{t \in [0, T]} |u(t)|^p \Big)}  \le \infty.
\end{align*}

\end{corollary}
\begin{lemma}\label{uboundd}
Stochastic process  $u(t)$ belongs to $Dom{\hspace{0.2mm}}\delta$.  
\end{lemma}
\begin{proof}
\noindent It is sufficient to show the random variables,
\begin{align*}
&F:={\int_{0}^{t} {\Big((Y_s^{x})^{-1}({{\alpha}_s ^{x}}- \sum_{n=1}^d {\beta _s ^{x,n}}{B_s ^{x,n}}-{\int_{{\mathbb{R}}_{0}}^{} {\frac{{M_{s,z}^x}{\gamma _s ^x}}{\mathbf{1}+{M_{s,z} ^x}}}\, {\mu (dz)}})\Big)}\, ds}, \notag\\
&G:=\sum_{n=1}^d \int_{0}^{t} {(Y_s^x)^{-1} {\beta _s ^{x,n}}}\, dW_s ^n, \notag\\
&K:=\int_{0}^{t} {\int_{\mathbb{R}_0}^{} {(Y_s ^{x})^{-1} {\frac{\gamma _s ^x}{\mathbf{1}+ {M_{s,z} ^{x}}}}}\, }\, {\tilde{N}}(dz,ds),
\end{align*}
belong to $Dom \delta$.
 \noindent Due to ${\mathbb{D}}^{1,2} {\subset} {Dom{\delta}}$ {\cite{nualart2018introduction}}, and also ${{\alpha}_s ^x}, {{\beta}_s ^{x,n}}, {{\mathit{B}}_s ^{x,n}}$ and ${\mathit{M}}_{s,z} ^x$ are Malliavin differentiable for every $\mathit{s} \in {{[0, T]}}$, it is enough to show that $\mathit{(Y_s ^x)}^{-1} \in {\mathbb{D}}^{1,2}$. Since $\mathit{(Y_s ^x)}^{-1} = \exp\{-H(s)\}$ and $H(s)$ is Malliavin differentiable with $D_{t,z}H(s)= ln(1+M_{t,z}^x)$, then from (\ref{equa1}) and (\ref{eq15}), 
 \begin{equation}\label{inverseddy}
 D_{t,z}\mathit{(Y_s^x)}^{-1}=\mathit{(Y_s ^x)}^{-1} \Big(e^{-D_{t,z}H(s)}-1\Big) \in L^2(P \times \mu),
\end{equation}
and therefore $\mathit{(Y_t ^x)}^{-1} \in {\mathbb{D}}^{1,2}$ for every $\mathit{t} \in {[0, T]}$ which comlpetes the proof. 
\end{proof}
\begin{corollary}\label{duboundd}
Stochastic process $u(t)$ is Malliavin differentiable and for every $p \geq 1$,
$$\sup_{0 \leq r \leq t \leq T, z \in \mathbb{R}_0}\mathbb{E}\Big(\vert D_{r,z}u(t) \vert^p\Big) < \infty.$$  
\end{corollary}
\begin{proof}
According to Proposition \ref{derjump}, the equations \eqref{inverseddy} and \eqref{equa1} complete the proof.
\end{proof}

\subsection{Relationship between $\mathit{D_{t,z}}$ and flow in semi-linear case}
\noindent In this part, we find a relationship between the Malliavin derivative of the mean-field stochastic differential equation driven by the jump process with its stochastic flow. Also, Malliavin derivative of the path-dependent payoff functions like $\max_{0 \le s \le T}X_t$ is established and used in sensitivity analysis of lookback options.\\
According to Proposition \ref{derjump}, for mean-field stochastic differential equation with jumps in (\ref{sde1}), we have
\begin{align*}
D_{r,z}{X_t ^x}&= \lambda(r, X_r ^x, z, \eta_{X_r ^x} ) + \int_{r}^{t}{b(s, \rho_{X_r ^x})D_{r,z}X_s ^x}ds \notag\\
& + \sum_{n=1}^{d}{\int_{r}^{t}{\big{\{}}{\sigma_n (s, X_s ^x + D_{r,z}X_s ^x, \pi_{X_s ^x}) - \sigma_n (s, X_s ^x, \pi_{X_s ^x})}{\big{\}}}dW_s ^n}\notag\\
& +\int_{r}^{t}{\int_{\mathbb{R}_0}^{}{\big{\{}}{\lambda(s, X_s ^x + D_{r,z}X_s ^x, z, \eta_{X_s ^x}) - \lambda(s, X_s ^x, z, \eta_{X_s ^x})}{\big{\}}}\tilde{N}(dz, ds)}.
\end{align*}
Now, consider the following semi-linear form of SDE,
\begin{align}\label{sde7}
X_t ^x = x + \int_{0}^{t}{b(s, \rho_{X_s ^x})X_s ^x} &+ \sum_{n=1}^d{\int_{0}^{t}\{{C_s^{n} X_s ^{x}+ {\sigma}_{0n} (s, {\pi}_{X_s ^x})}\}d{W_s ^n}}\notag\\
&+\int_{0}^{t}{\int_{\mathbb{R}_0}^{}{\{{F_{s,z} X_s^{x} + \lambda_0(s, z, \eta_{X_s ^x})\}\tilde{N}(dz,ds)}}}.
\end{align}
From Assumption \ref{assum1} and relation (\ref{tilde}), for each $T>0$ and $p\ge2$ there exists $R^\prime = R^\prime (T, p)$, such that for n=1, ..., d,
\begin{equation*}
|xf^{-1}(t)C_t ^n| + \int_{\mathbb{R}_0}^{}{|xf^{-1}(t)F_{t,z}|}\mu(dz) \le R^\prime,
\end{equation*}
\begin{align*}
&\sum_{n=1}^{d}{\Big{(}\int_{0}^{T}{|xf^{-1}(s)\sigma_{0n}(s, \pi_{X_s ^x})|^2 ds}\Big{)}^{p/2}}  \notag\\
&+ \Big{(}{\int_{0}^{T}{\int_{\mathbb{R}_0}^{}{|xf^{-1}(s) \lambda_{0}(s, z, \eta_{X_s ^x})|^2}}\mu(dz) ds}\Big{)^{p/2}} \le R^\prime,
\end{align*}
and emphasize that from \eqref{eq1} and \eqref{eq2} there exist some constants $\alpha_F$ and $d_F$ that   
\begin{equation}\label{alphaf}
\alpha_F < \sup_{0 \leq r \leq T, z\in \mathbb{R}_0}  \vert F_{r,z} \vert < d_F.
\end{equation}
Therefore, the conditions of Lemma 5.19 in \cite{menaldi2008stochastic} hold, $S_t$ is invertible and 
\begin{equation*}
\mathbb{E}\Big{(}{{{\sup_{0 \le t\le T} |{\frac{1}{{S_t ^x}}}|^p} }}\Big{)} < \infty.
\end{equation*}
Consequently, 
\begin{equation}\label{inversexp}
\mathbb{E}\Big( {\sup_{0 \le t\le T} |{\frac{1}{{X_t ^x}}}|^p} \Big)  \le e^{-\int_0^T b(s,f(s) ds}\mathbb{E}\Big{(}{{\sup_{0 \le t\le T} |{\frac{1}{{S_t ^x}}}|^p} }\Big{)}< \infty.
\end{equation}
\noindent By Proposition \ref{derjump} Malliavin derivative for $r<t$ satisfies 
\begin{align*}
D_{r,z}{X_t ^x}
&={\lambda (r,X_r ^x, z, {\eta}_{X_r ^x})} + {\int_{r}^{t} {b (s,{\rho}_{X_s ^x}) D_{r,z}{X_s ^x}}\, ds} +\sum_{n=1}^d {\int_{r}^{t} {C_s ^{n}D_{r,z} {X_s ^x}}\, d{W_s ^n}} \notag\\
&+{\int_{r}^{t} {\int_{\mathbb{R}_{0}}^{} {F_{s,z} D_{r,z} {X_s ^x}}\, }\, \tilde{N} (dz,ds)}.
\end{align*}
\begin{corollary}
In a manner analogous to what we have established for $X_t$, we can demonstrate there exists some process $S_t ^*$ such that $\mathbb{E}\big(\sup_{t \in [0,T]} S_t^*\big) < \infty$, and 
\begin{align*}
{{D_{r,z}X_t} ^{{x}} := exp\{{\int_{0}^{t}b(s, f(s))ds}}\}S_t ^*.
\end{align*}
 By applying similar reasoning as that used for $X_t$, in conjunction with the results from equation (\ref{eq1}), the linear growth property of $\lambda$ and \eqref{equasup}, we can assert that for any  $p \ge 1$, the following conclusion is valid.
\begin{align*}
\mathbb{E} {\Big(\sup_{t \in [0, T]} |D_{r,z} X_t ^x|^p \Big)}  < \infty.
\end{align*}
\end{corollary}
\noindent On the other hand, according to the definition of $Y_t$, we deduce
\begin{equation}\label{eqq}
D_{r,z} {X_t ^x} = {Y_t ^x} {(Y_r ^x) ^{-1}} {\lambda (r,X_r ^x, z, {\eta}_{X_r ^x})} {\mathbf{1}_{r \le t}}.
\end{equation}
As a consequence, from (\ref{eq7}), for every $r,t {\in} [0,T]$ with $ r{\le}t$,
we derive desired relationship between the stochastic flow and the Malliavin derivative of $X_t ^x$.\\
\begin{equation*}
{\frac{\partial}{\partial{x}}{X_t ^x }}={D_{r,z}}{X_t ^x } {\lambda^{-1} (r,X_r ^x, z, {\eta}_{X_r ^x})} {Y_r ^x}{{u}(t)}, \hspace{7mm} r{\le}t.
\end{equation*}
\noindent Here, we demonstrate the Malliavin differentiability of the function $G_M(X^x)=\max_{0 \leq t \leq T}X_t^x$ which serves for pricing the barrier-lookback options. 
\begin{lemma}
 The random variables $G_M(X^x)=\max_{0 \leq t \leq T}X_t^x$ and $G_m(X^x)=\min_{0 \leq t \leq T}X_t^x$ are Malliavin  differentiable.
\end{lemma}
\begin{proof}
Let $\left(t_{k}\right)_{k \geq 1}$ be a dense subset of $[0, T]$, and for each $n \geq$ 1, define $G_{n}=\max \left\{X_{t_{1}}^x, \ldots, X_{t_{n}}^x\right\}$. Clearly, $\left(G_{n}\right)_{n>1}$ is an increasing sequence bounded by $G_M(X^x)$. Hence $G_{n}$ converges to $G_M(X)$ in the $L^{2}(P)$-norm when $n$ goes to infinity. Also, since $G_n$ is a Lipschitz function on $\mathbb{R}^n$, 
$$D_{t,z}G_n = \max\left\{X_{t_{1}}^x+D_{t,z}X_{t_{1}}^x1_{t \leq t_1}, \ldots, X_{t_{n}}^x+D_{t,z}X_{t_{n}}^x1_{t \leq t_n}\right\}-G_n,$$   
and 
$$\sup_{n \geq 1} \Vert DG_n \Vert_{L^2(P\times \lambda \times \mu)}\leq \sup_{t \leq t_i \leq T}\Vert D_{t,z}X_{t_i}^x1_{t \leq t_i}\Vert_{L^2(P\times \lambda \times \mu)} < \infty.$$
Consequently, $G(X^x)$ belongs to $\mathbb{D}^{1,2}$ and from Proposition \ref{pro23} we derive
\begin{equation*}
D_{t,z}G_M(X^x)=\sup_{0 \le s\le T}\left\{X_{s}^x+D_{t,z}X_{s}^x1_{t \leq s}\right\}-G_M(X^x).
\end{equation*}
Similarly, for $G_m(X^x):=\min_{0 \le t \le T}X_s^x$ is Malliavin differentiable and 
\begin{equation}\label{dminx}
D_{t,z}\min_{0 \le t \le T}X_s^x=\inf_{0 \le s\le T}\left\{X_{s}^x+D_{t,z}X_{s}^x1_{t \leq s}\right\}-\min_{0 \le t \le T}X_s^x.
\end{equation}
\end{proof}

\section{Applications}
\noindent Malliavin calculus offers significant advantages for sensitivity analysis, particularly in computing the Delta for mean-field stochastic differential equations with jumps. A foremost benefit is its capacity to significantly reduce variance in Monte Carlo simulations, which is especially advantageous when addressing discontinuous payoff functions. This reduction in variance results in more stable, as mentioned in \cite{fournie1999applications}, and precise outcomes compared to conventional finite difference methods. Its computational efficiency is particularly notable in high-dimensional contexts, where it facilitates the calculation of sensitivities without necessitating multiple simulations, thereby conserving computational resources. Additionally, the theoretical rigor inherent in Malliavin calculus enhances the precision of sensitivity analyses, providing deeper insights into the structure of stochastic processes. These attributes collectively render it an indispensable and powerful tool for tackling complex stochastic dynamics and improving risk management strategies in financial systems.\\
\noindent Now, we can present an application of the above results in finding an expression for Delta of the square-integrable pay-off function $\Phi (X^x_T, G(X^x))$ with respect to a Malliavin weighting function which has a Skorokhod integral form, where $X_t ^x$ be a mean-field SDEs driven by jump process of the form (\ref{sde1}). We determine which conditions these functions should satisfy to serve as weighting functions. We will also find this weight function for European call options and barrier options.\\
Put $G(X^x)=G_{i}(X^x)$ for $i=M, m$ and define $\vartheta (x)$ in the following way 
\begin{equation*}
{\vartheta}(x)=\mathbb{E}{\vartheta}_{X,G}(x):= {\mathbb{E}}\Big({\Phi}(X_T ^x, G(X^x)) \Big).
\end{equation*}
\begin{theorem}\label{th1}
Let ${\Phi} :{\mathbb{R}^2}{\to}{\mathbb{R}}$ be of polynomial growth. Then $\vartheta$ is continuously differentiable and if there exists a Skorohod integrable function $\omega (.,.)$ satisfying the following equation
\begin{align*}
\mathbb{E} &\Big( {\Phi}_{X}^{\prime}(X_T ^x, G(X^x)) {Y_T ^x} u(T)+ {\Phi}_G^{\prime}(X_T ^x, G(X^x)) \frac{\partial }{\partial x}G(X^x)\Big)\\
&=  {\mathbb{E}} {\Big( } {\int_{0}^{T}{\int_{{\mathbb{R}}_0}^{}{{\omega}(t,z) {\Big( }{\Phi}(X_T^x + D_{t,z} X_T^x, G(X^x)+ D_{t,z}G(X^x)) - \vartheta_{X,G} {\Big) }}{\mu}(dz)}dt} {\Big) },
\end{align*}
where ${\Phi}_{X}^{\prime}$ and ${\Phi}_G^{\prime}$ are derivative of $\Phi$ with respect to the first and the second components, respectively, then
\begin{align*}
 &\Delta = {\frac{\partial}{\partial x}}{\vartheta}(x) = \mathbb{E} \Big( {\Phi}(X_T^x, G(X^x)) \delta (\omega) \Big).
\end{align*}
\end{theorem}
\begin{proof}
\noindent According to the approach used in \cite{huehne2005malliavin}, 
\begin{align}
\Delta = \frac{\partial}{\partial x} \vartheta &= \frac{\partial}{\partial x}{\mathbb{E}} \Big( \Phi (X_T ^x, G(X^x))  \Big) \notag\\
&=\mathbb{E} \Big( {\Phi}_{X}^{\prime}(X_T ^x, G(X^x)) {Y_T ^x} u(T)  + {\Phi}_G^{\prime}(X_T ^x, G(X^x))\frac{\partial }{\partial x}G(X^x)\Big), \label{Delta}
\end{align}
\noindent where the interchange of the derivative and the expectation operator is justified by the dominated convergence theorem. Now, from the assumption and the duality formula in Poisson space we have
\begin{align*}
&\Delta = \mathbb{E} \Big( \Phi (X_T ^x, G(X^x)) \int_{0}^{T}{\int_{\mathbb{R}_0}^{}{\omega (t,z) \tilde{N} (dz, \delta t) }} \Big) \notag\\
&= \mathbb{E} \Big( \int_{0}^{T}{\int_{\mathbb{R}_0}^{}{\omega (t,z) D_{t,z} (\Phi (X_T ^x, G(X^x))) \mu (dz) dt}}  \Big) \notag\\
& = \mathbb{E} \Big( \int_{0}^{T}{\int_{\mathbb{R}_0}^{}{\omega (t,z) \big( \Phi ( X_T^x + D_{t,z}X_T ^x,G(X^x)+ D_{t,z}G(X^x)) - \vartheta_{X,G}  \big) \mu(dz) dt}} \Big).
\end{align*}
\end{proof}
\begin{example}{European call option}{\label{Example5.2}}\\
\noindent The pay-off function of a standard European call option depends on the underlying stock price at the expiry time and the strike price. Since the pay-off function only depends on the underlying stock price at maturity, it is usually referred to as a path-independent option. For the European call option with a pay-off function 
\begin{equation*}
\Phi (X_T) = (X_T - K)^{+},
\end{equation*}
 and $X_t^x$ in the semi-linear equation (\ref{sde7}), we state the weighting function of the delta as follows
\begin{align}\label{wEur}
\omega (t,z)= \frac{H_K (X_T)Y_T u_T}{\mathcal{Q} \Big[(X_T + D_{t,z} X_T - K)^{+} - (X_T - K)^{+}\Big]},
\end{align}
where $\mathcal{Q}:= \int_{0}^{T}\int_{\mathbb{R}_0}^{}\mu(dz)dt$, and $\mathit{H_y (x)}=\mathbf{1}_{x>y}$ is a Heaviside function
which implies that for ${X_T - K >0}$ 
\begin{align}\label{omega1}
\omega (t,z)=
\begin{cases}
\frac{Y_T u_T}{\mathcal{Q}{D_{t,z} X_T}} \hspace{7mm}{D_{t,z} X_T \ge K - X_T}, \\
\\
\frac{Y_T u_T}{\mathcal{Q}{(K-X_T)}}\hspace{7mm} {D_{t,z} X_T < K - X_T},
\end{cases}
\end{align}
and ${0}$ otherwise. 
It is essential to show $\omega \in Dom\delta$. 
\begin{lemma}\label{lemm1}
If  the function $\lambda^{-1}$ is $L^p$-integrable, for $p \geq 1$, i.e, $$\mathbb{E}\big(\sup_{0 \leq t \leq T}\vert \lambda(t, X_t ^x, z, \eta_{X_t ^x}\big) \vert^{-p}) < \infty,$$ then the function $w$ belongs to $Dom\delta$.
\end{lemma}
\begin{proof}
From Lemma \ref{uboundd} and Corollary \ref{duboundd}, it is sufficient to show that $\omega_0^1:=\frac{Y_T}{D_{t,z}{X_T ^x}} \in  \mathbb{D}^{1, p}$  and $\omega_0^2:=\frac{Y_T}{(K- X_T ^x)} \in  \mathbb{D}^{1, p}$. 
In the light of relation (\ref{eqq}), we have
\begin{equation*}
w_0^1= \frac{1}{\mathcal{G}}:=\frac{1}{Y_t ^{-1} \lambda(t, X_t ^x, z, \eta_{X_t ^x})}.
\end{equation*}
The fact that $Y_t ^{-1}$ and $X_t ^x$  are Malliavin differentiable implies that $w_0^1, w_0^2$ is also Malliavin differentiable. 
 Based on \eqref{inversexp} and \eqref{alphaf}, when $\frac1p+\frac1q=1$, 
\begin{align*}
\mathbb{E}\Big{(}\int_0^T\int_{\mathbb{R}_0}& |w_0^1(t,z)|^p\mu(dz)dt\Big{)}\\
&\le {\mathbb{E}}^{\frac1q}\big{(}\sup_{0 \le t \leq T} |Y_t|^q\big{)}\Big{(}\int_0^T\int_{\mathbb{R}_0}\mathbb{E}^{\frac1p}\Big\vert \lambda(t, X_t^x, z,\eta_{X_t^x}) \Big\vert^{-p}\mu(dz)dt\Big{)} 
< \infty.
\end{align*}
 Taking derivative of $w_0$ for every $0 \leq r \leq t$ and $z' \in \mathbb{R}_0$, and apply \eqref{inverseddy} to achieve 
\begin{align*}
\vert D_{r,z'}w_0^1 (t,z) \vert & = \vert \frac{-D_{r,z'}\mathcal{G}}{\mathcal{G}(\mathcal{G}+D_{r,z'}\mathcal{G})}\vert  \leq \vert \frac{1}{\mathcal{G}} + \frac{1}{D_{r,z'}\mathcal{G}} \vert \\
& = \vert w_0^1\vert + \vert  \frac{Y_t^2(1+M_{r,z'})}{F_{t,z}(X_t+D_{r,z'}X_t)}\vert\\
& = \vert w_0^1\vert + \frac12 \vert  \frac{Y_t^2(1+M_{r,z'})}{F_{t,z}}\vert \vert \frac{1}{X_t} + \frac{1}{D_{r,z'}X_t}\vert.
\end{align*}
Our assumption, \eqref{eqq} and \eqref{equa1} guarantee that 
$$\mathbb{E}\Big{(}\int_0^T\int_{\mathbb{R}_0} \vert \frac{1}{X_t} \vert^p \mu(dz)dt+\int_r^T\int_{\mathbb{R}_0}\vert \frac{1}{D_{r,z'}X_t}\vert^p \mu(dz)dt\Big) < \infty.$$ 
Finally, utilize \eqref{ytp} and \eqref{alphaf} to derive that $w_0^1 \in  Dom\delta$.\\
The same attribute as above is valid to show $w_0^2$ belongs to $\mathbb{D}^{1, p}$ and we need to utilize \eqref{inversexp} in the same way.
\end{proof}
\noindent We note here that from \eqref{inversexp}, if $\lambda_0(t, z,\eta_{X_t^x}) =0$, the condition in Lemma \ref{lemm1} satisfies immediately.
\end{example}
\begin{example}{Barrier knock out option }
\subsubsection*{Down and out call}
\noindent As long as the underlying stock price remains above the barrier price, the option offers a pay-off function equivalent to a European call option. If the underlying stock price drops to or below the barrier price at any time during the option lifetime, the call option is immediately terminated. This can be written as {\cite{schoutens2003pricing}}
\begin{equation*}
\Phi (X_{\tau}) = (\min_{0 \le t \le T}X_t^x- K)^{+} \mathbf{1}_{(X_{\tau} > B)}= (\min_{0 \le t \le T}X_t^x-K)^+\mathbf{1}_{\min_{0 \le t \le T}X_t^x> B},
\end{equation*}
where $X_t$ is the underlying stock price at time $t$, $K$ is the strike price, $B$ is the barrier price agreed at the beginning of the contract, and $\tau$ is the stopping time that the process $X^x$ reaches its minimum on $[0, T]$. Apply Proposition \ref{pro23} to result $$D_{t,z}\mathbf{1}_{G_m(X^x)> B}= \mathbf{1}_{G_m(X^x)+D_{t,z}G_m(X^x)> B}-\mathbf{1}_{G_m(X^x)> B}.$$
Now, Proposition \ref{pro22} and \eqref{dminx} deduce that
\begin{align*}
D_{t,z}\Phi(X_\tau)&=(G_m(X^x)-K)^+ D_{t,z}\mathbf{1}_{G_m(X^x)> B} +D_{t,z}(G_m(X^x)-K)^+\mathbf{1}_{G_m(X^x)> B}\\
&+D_{t,z}(G_m(X^x)-K)^+D_{t,z}\mathbf{1}_{G_m(X^x)> B}\\
&= (G_m(X^x)+D_{t,z}G_m(X^x) -K)^+\mathbf{1}_{G_m(X^x)+D_{t,z}G_m(X^x)> B}\\
&-(G_m(X^x)-K)^+\mathbf{1}_{G_m(X^x)> B}\\
&=\Phi ((X^x+\mathbf{1}_{t \leq .}D_{t,z}X^x)_{\tau^*})-\Phi(X_\tau),
\end{align*}
in which  $\tau^*$ is the stopping time that the process $X_s^x+\mathbf{1}_{t \leq s}D_{t,z}X_s^x$ reaches its minimum. 
By determine $\tau$ and $\tau^*$, we can provide barrier Malliavin weights as follows,
\begin{align*}
\omega_{Br} (t,z)= \frac{\mathbf{1}_{X_\tau> B} H_K (X_{\tau})\frac{\partial}{\partial x}(G_m(X^x)) }{\mathcal{Q}\Big[\Phi ((X^x+\mathbf{1}_{t \leq .}D_{t,z}X^x)_{\tau^*})-\Phi(X_\tau)\Big]}.
\end{align*}
We can conclude that $\omega \in Dom(\delta)$, similar to the proof of Lemma \ref{lemm3} in the case of up and out barrier options, so we omit it here and just prove it in the later case.
\subsubsection*{Up and out call}
\noindent  The pay-off function of this option is equivalent to that of a European call option if the underlying stock price remains below the barrier price. If the underlying stock price rises above the barrier price during the life of the option, the call option is immediately exercised. Here, $\tau$ and $\tau^*$ are the stopping times that the processes $X_s^x$ and $X_s^x+\mathbf{1}_{t \leq s}D_{t,z}X_s^x$ reaches their maximum, respectively. The pay-off function can be written as
\begin{equation*}
\Phi (X_{\tau}) = (\max_{0 \le s \le T}X_s^x - K)^{+} \mathbf{1}_{(X_{\tau} < B)},
\end{equation*}
and the Malliavin weight is obtained as follows
\begin{align}\label{W2B}
\omega_{Br} (t,z)= \frac{\mathbf{1}_{X_\tau< B} H_K (X_{\tau})\frac{\partial}{\partial x}(G_M(X^x))}{\mathcal{Q}\Big[\Phi ((X_T^x+\mathbf{1}_{t \leq T}D_{t,z}X_T^x)_{\tau^*})-\Phi(X_\tau)\Big]},
\end{align}
We should just modify that $\omega \in Dom(\delta)$. To do this, let    
\begin{equation*}
\mathcal{B}_3= \{K < G_M(X^x+D_{t,z}X^x ))< B\}, \quad  \mathcal{B}_4= \{ G_M(X^x+D_{t,z}X^x ))< B \wedge K\},
\end{equation*}
\begin{equation*}
w_{0,3}(t,z):= \frac{1_{\mathcal{B}_3}}{\mathcal{G}_3}:=\frac{1_{\mathcal{B}_3}}{G_M(X^x+D_{t,z}X^x )-G_M(X^x)}, 
\end{equation*}
and
\begin{equation*}
 w_{0,4}(t,z):= \frac{1_{\mathcal{B}_4}}{\mathcal{G}_4}:=\frac{1_{\mathcal{B}_4}}{K- G_M(X^x)} .
 \end{equation*}
 Therefore,
\begin{equation*}
\omega_{Br}(t,z)=\frac{ 1_{K < G_M(X^x))< B}\frac{\partial}{\partial x}(G_M(X^x))}{\mathcal{Q}}[1_{\mathcal{B}_3}w_{0,3}(t,z)+ 1_{\mathcal{B}_4}w_{0,4}(t,z)].
\end{equation*}
We assume the following condition in the sequence.
\begin{assumption}\label{assume3}
that there exists some $\epsilon_0 > 0$ such that 
\begin{equation*}
P\Big(\inf_{t,z}\inf_{r,z'} \big[\vert \mathcal{G}_3\vert +\vert \mathcal{G}_3+ D_{r,z'}\mathcal{G}_3 )\vert\big]> \epsilon_0\Big) =1.
\end{equation*} 
\end{assumption}
\begin{lemma}\label{lemm3}
Under Assumption \ref{assume3}, if  the function $\frac{\partial}{\partial x}(G_M(X^x))$ belongs to $\mathbb{D}^{1,p}$, for $p \geq 1$, i.e, 
$$\mathbb{E}\Big(\sup_{0 \leq r \leq t \leq T}\vert D_{r,z}\frac{\partial}{\partial x}(G_M(X^x))  \vert^{p}\Big) < \infty,$$ then the function $w$ belongs to $Dom\delta$.
\end{lemma}
\begin{proof}
Accodring to Corollary \ref{coro}, 
$$\mathbb{E}\Big(\sup_{0 \leq t \leq T}\vert \frac{\partial}{\partial x}(G_M(X^x))  \vert^{p} \Big) \leq \mathbb{E}\Big(\sup_{0 \leq t \leq T}\vert \frac{\partial}{\partial x}X^x  \vert^{p}\Big) < \infty.$$
So, it is sufficient to show that $w_{0,3}$ and $w_{0,4}$ belong to $Dom\delta$.
We first note that for every $0 \leq r \leq t$ and $z'\in \mathbb{R}_0$
\begin{align*}
D_{r,z'} w_{0,3} (t,z)& = 1_{\mathcal{B}_3} D_{r,z'} \frac{1}{\mathcal{G}_3} + \frac{1}{\mathcal{G}_3} D_{r,z'} 1_{\mathcal{B}_3}+D_{r,z'} 1_{\mathcal{B}_3} D_{r,z'} \frac{1}{\mathcal{G}_3}\\
& =  - \frac{(1_{\mathcal{B}_3}+D_{r,z'}1_{\mathcal{B}_3}) D_{r,z'}\mathcal{G}_3}{\mathcal{G}_3^2+\mathcal{G}_3 D_{r,z'}\mathcal{G}_3 }+  \frac{1}{\mathcal{G}_3} D_{r,z'} 1_{\mathcal{B}_3},
\end{align*}
and therefore 
\begin{align*}
\vert D_{r,z'} w_{0,3} (t,z) \vert^p &\leq 2^{p}  \Big( \vert \frac{1}{\mathcal{G}_3}+ \frac{1}{\mathcal{G}_3+D_{r,z}\mathcal{G}_3}\vert^p + \vert \frac{1}{\mathcal{G}_3}\vert^p \Big)\\
& \leq 2^{p-1} \Big(  \vert \frac{2}{\mathcal{G}_3}\vert^p+\vert \frac{1}{\vert D_{r,z}\mathcal{G}_3}\vert^p + \vert \frac{1}{\mathcal{G}_3}\vert^p \Big),
\end{align*}
and Assumption \ref{assume3}, conclude that $w_{0,3} \in Dom\delta$. Similarly, one can show that $w_{0,4} \in Dom\delta$ and complete the proof.
\end{proof}
 \end{example}
\section{Numerical Scheme and the convergence of Delta}
\noindent Significantly, the Euler scheme is employed to simulate $X_t ^x$ in the implementation of numerical experiments. Therefore, before presenting the numerical results, we provide a formal analysis of the convergence of Euler method.\\
\noindent Let $X^{n,x}_t$ represents the solution derived from the Euler scheme for any $n \in \mathbb{N}$ and partition $[0, T]$ with $t_i=\frac{it}{n}$, $i=1, \cdots, n$, as demonstrated in the following
\begin{align}\label{Eulersde}
X_{t_{k+1}}^{n,x} = X_{t_k}^{n,x} &+ b(t_k, \rho_{t_k} )X_{t_k}^{n,x} \Delta t   
+ \sum_{n=1}^d \sigma_n(t_k, X_{t_k}^{n,x}, \pi_{t_k}) \Delta W_k^n \notag \\
&+ \sum_{j=1}^{\mathcal{N}_k} \lambda(t_k, X_{t_k}^{n,x}, z_j, \eta_{t_k}),  
\end{align}
\noindent where $\Delta W_k^n = W_{t_{k+1}}^n - W_{t_k}^n$ representing the Brownian increments, and $\mathcal{N}_k \sim \text{Poisson}(\nu \Delta t)$ is the number of jumps in $[t_k, t_{k+1}]$ and $z_j$ is the jump size in $j$-th jump. 
For every $t \in [t_k, t_{k+1}]$, define the continuous version 
\begin{align*}
X_{t}^{n,x} = X_{t_k}^{n,x} &+ \int_{t_k}^t b(t_k, \rho_{t_k} )X_{t_k}^{n,x} ds  
+ \sum_{n=1}^d  \int_{t_k}^t \sigma_n(t_k, X_{t_k}^{n,x}, \pi_{t_k}) dW_s^n \notag \\
&+  \int_{t_k}^t \int_{\mathbb{R}_0} \lambda(t_k, X_{t_k}^{n,x}, z_k, \eta_{t_k}){\tilde{N}}(dz_k, ds).
\end{align*}
Let
\begin{equation}\label{Yyn}
d{Y_t ^{n,x}}={A_t ^{n,x}}Y_t ^{n,x} dt + \sum_{m=1}^d B_t ^{n,x,m}Y_t ^{n,x}d{W_t ^{m}} + \int_{\mathbb{R}_{0}}^{} {M_{t,z}^{n,x}}Y_t ^{n,x}{\tilde{N}}(dz, dt), 
\end{equation}
\begin{align*}
du^n(t)={A_t ^{n,*}} {(Y_t ^{n,x})}^{-1} dt + \sum_{m=1}^{d} \beta_t^{n,x,m}{( Y_t ^{n,x})}^{-1} d{W_t ^m} + \int_{{\mathbb{R}}_{0}}^{}  M_{t,z} ^{n,*}{(Y_t ^{n,x})}^{-1}\,{\tilde{N}}(dz,dt),
\end{align*}
where $A_t ^{n,x}, B_t ^{n,x,m}, M_{t,z}^{n,x},$ and $A_t ^{n,*}, \beta_t^{n,x,m}, M_{t,z} ^{n,*}$ are the same as $A_t ^{x}, B_t ^{x,m}, M_{t,z}^{n,x},$ and $A_t ^{*}, \beta_t^{x,m}, M_{t,z} ^{*}$, respectively, when we substitute the $ X_t^{n,x}$ instead of $X_t^x$.
\noindent In \cite{sun2021numerical}, Sun and Zhao proved that, under assumptions  A1–A4, the strong convergence order of the strong order-$\gamma$ Taylor scheme is $\gamma$, specifically established
\begin{align} \label{Taylor} 
\max_{k=1, \cdots , N} \mathbb{E}\Big(|S^x_{t_k} - {S}_{t_k}^{n,x}|^{2\gamma}\Big)   
&\leq C(\Delta t)^\gamma, \quad \max_{k=1, \cdots , N} \mathbb{E}\Big(| {S}_{t_k}^{n,x}|^{2\gamma}\Big)   
&\leq C.  
\end{align} 
In the following, we will analyze the strong convergence order of $X_t$ under Assumption \ref{assum1}, utilizing the results of \cite{sun2021numerical}.
\begin{theorem}\label{thmthm}
\noindent Under Assumption \rm{\ref{assum1}}, the Euler scheme $X_{t}^{n}$ converge to the true solution with order $\frac12$, i.e., it holds that for every $p \geq 2$ 
\begin{align}\label{Eulerorder}
 \mathbb{E}\left(\sup_{0 \le t \leq T}  \lvert X_{t}^x - {X}_{t} ^{n,x} \rvert^{2} \right)   
\leq C  (\Delta t). 
\end{align}
\end{theorem}

\begin{proof}
\noindent Building upon the definition of $S_t$ and the relation \eqref{Taylor} there exists some constant $C_1$ that
\begin{align}
	\mathbb{E}\Big({\big|{S_t ^n -S_{t_k} ^n}\big|^2}\Big) & = \mathbb{E}\Big(\big|{\int_{t_k}^{s}{\tilde{\sigma}(t_k, S ^n_{t_k}, \pi_{S ^n_{t_k}})}}\big|^2\Big) + \mathbb{E}\Big(\big|{\int_{t_k}^{s}{\tilde{\lambda}(t_k, S ^n_{t_k}, z, \eta_{S_{t_k}})}}\big|^2\Big)\notag\\
	& \le \mathbb{E}\Big({|\tilde{\sigma}(t_k, S ^n_{t_k}, \pi_{S ^n_{t_k}})|^2 (W_s - W_{t_k})^2}\Big) \notag\\
	&+ \mathbb{E}\Big({|\tilde{\lambda}(t_k, S ^n_{t_k}, z, \eta_{S ^n_{t_k}})|^2 \vert\int_{t_k}^{s}\int_{\mathbb{R}-0}{\tilde{N}(dz, dt)\vert^2}}\Big)\notag\\
	& \le C_1 (\Delta t)^2 + C_1 (\Delta t). \label{aaa}
\end{align}  
Fom Lipschitz property of $\tilde\sigma$ and $\tilde\lambda$, alongside the relation \eqref{Taylor} and \eqref{aaa} for some constant $C_2$
\begin{align*}  
&\mathbb{E}\Big(\sup_{t_k \le t \leq t_{k+1}}\lvert S_{t} - S_{t}^{n} \rvert^{2} \Big) \notag\\
& \le \mathbb{E}\Big(\vert \sup_{t_k \le t \le t_{k+1}}\int_{t_k}^{t} \big| \tilde\sigma(s, S_s, \pi_{S_s}) - \tilde\sigma(t_k, S_{t_k}^{n}, \pi_{S_{t_k}^{n}}) \big| dW_s \vert^2\Big) \notag\\
& + \mathbb{E}\Big(\sup_{t_k \le t \le t_{k+1}}\big\vert \int_{t_k}^{t}\int_{\mathbb{R}_0} \big| \tilde\lambda(s, S_s, z, \eta_{S_s}) - \tilde\lambda(t_k, S_{t_k}^{n}, z, \eta_{S_{t_k}^{n}}) \big| \tilde{N}(dz, ds)\big\vert ^{2}\Big) \notag\\
& \le \mathbb{E}\Big(\int_{t_k}^{t_{k+1}}\big| \tilde\sigma(t, S_s, \pi_{S_s}) - \tilde\sigma(t_k, S_{t_k}^{n}, \pi_{S_{t_k}^{n}}) \big|^{2} ds\Big) \notag\\
& + \mathbb{E}\Big(\int_{t_k}^{t_{k+1}}\int_{\mathbb{R}_0} \big| \tilde\lambda(s, S_s, z, \eta_{S_s}) - \tilde\lambda(t_k, S_{t_k}^{n}, z, \eta_{S_{t_k}^{n}}) \big|^{2} \mu(dz) ds\Big) \notag\\
& \le 2\int_{t_k}^{t_{k+1}}\Big[ \mathbb{E}\Big(\sup_{0 \le u \le s}\big| S_u - S_u ^n \big|^{2}\Big)+\mathbb{E}\Big(\big| S_s^n - S_{t_k}^n \big|^{2}\Big) + \mathbb{E}\Big(\big| S_s - S_s ^n  \big|\Big)^{2} \Big]ds \notag\\
& + 2\int_{t_k}^{t_{k+1}}\int_{\mathbb{R}_0}\Big[\mathbb{E}\Big( \sup_{0 \le u \le s} \big| S_u - S_u ^n \big|^{2}\Big)+\mathbb{E}\Big(\big| S_s^n - S_{t_k}^n \big|^{2}\Big)  + \mathbb{E}\Big(\big| S_s - S_s ^n  \big|\Big)^{2}\Big] \mu(dz) ds\notag\\
& \le C_2 + 8{\int_{t_k}^{t_{k+1}}\mathbb{E}\Big( \sup_{0 \le u \le s} \big| S_u - S_u ^n \big|^{2}\Big) ds},
\end{align*}
\noindent utilizing Gronwall inequality, we can deduce that
\begin{align*}
\mathbb{E}\Big(\sup_{0 \le t \leq T}\lvert S_{t} - S_{t}^{n} \rvert^{2} \Big) \to 0, \quad \quad n\to \infty.
\end{align*}

\noindent Consequently, from Theorem \ref{XT}, we can now proceed as
\begin{align*}  
	\mathbb{E}\left( \sup_{0 \le t \leq T}\lvert X_{t}^x - {X}_{t} ^{n,x} \rvert^{2} \right)  
	&= \mathbb{E}\left( \sup_{0 \le t \leq T}\left\lvert e^{\int_0^{t} b(s, \rho_{X_s}) \, ds} S_{t}   
	- e^{\int_0^{t} b(s, \rho_{{X}_s ^n}) \, ds } {S}_{t} ^n \right\rvert^{2} \right) \notag \\
	&= \mathbb{E}\Big(\sup_{0 \le t \leq T} \Big\lvert e^{\int_0^{t} b(s, \rho_{X_s}) \, ds} S_{t} - e^{\int_0^{t} b(s, \rho_{X_s}) \, ds} S_{t} ^n \notag\\
	&\quad\quad\quad + e^{\int_0^{t} b(s, \rho_{X_s}) \, ds} S_{t} ^n  - e^{\int_0^{t} b(s, \rho_{{X}_s ^n}) \, ds } {S}_{t} ^n \Big\rvert^{2} \Big) \notag \\
	&\le\Big( e^{{2} \int_0^{t} b(s, \rho_{X_s}) \, ds }\Big)\mathbb{E}\Big(\sup_{0 \le t \leq T}\lvert S_{t} - {S}_{t}^{n} \rvert^{2} \Big) \notag \\
	&\quad \quad \quad + \mathbb{E}\Big(\sup_{0 \le t \leq T} \lvert {S}_{t} ^n\rvert^{2}\Big) \Big \lvert e^{\int_0^{t} b(s, \rho_{X_s}) \, ds }- e ^{\int_0^{t} b(s, \rho_{{X}_s ^n}) \, ds} \Big\rvert^{2}. \notag\\
\end{align*}
Assuming the Lipschitz continuity of \(b(t, \rho_{X_t^x})\), the relation (\ref{Taylor}), invoking Lemma \ref{t1}, the inequality \(|e^x - e^y| \leq |e^x + e^y||x - y|\), and the boundedness of \(f(t)\), we derive the following
\begin{align*}
\mathbb{E}\left(\sup_{0 \le t \leq T} \lvert X_{t}^x - {X}_{t} ^{n,x} \rvert^{2} \right) &\le C_1 (\Delta t) + C_1 \Big| {\int_{0}^{t} b(s, \rho_{{X}_s}) \, ds} - {\int_{0}^{t} b(s, \rho_{{X}_s ^{n}}) \, ds}\Big|^{2}\notag\\
&\le C_3 (\Delta t) + C_3 (\Delta t)^{2},
\end{align*}
\noindent  where $C_3$ is a constant. This inequality completes the proof.
\end{proof}
\noindent 
We have established the validity of statement $w(t, z)$, $X_t$, and $D_{r,z} X_t$ belong to $L^p (P)$. Similarly, we can show that $w^n(t, z)=w(t,z)\Big\vert_{X_t^x= X_t^{n,x}}$, $X^{n,x}_t$, $\frac{\partial}{\partial x}X_t^{n,x}$and $D_{r,z} X^{n,x}_t$ belong to $L^p (P)$, for any $p \geq 2$.\\
Additionally, based on our discussion, $D_{r,z} X^{n,x}_t$ and $Y_t^{x,n}$ are solutions to (\ref{Yyn}) with the initial conditions $\lambda(r, X^{n,x}_r, z, \eta_{X^{n,x}_r})$ and $1$, respectively. Therefore, by applying the same proof of Theorem \ref{thmthm} we can conclude that
\begin{align}  
&\max_{k \in \{1, 2, \dots, N\}} \mathbb{E}\Big( \lvert D_{r,z} X^x_{t_k} - D_{r,z} X_{t_k}^{n,x} \rvert^2 \Big)   
\leq C \Big( 1 + \mathbb{E}\Big(\lvert \lambda(r, X_r, z, \eta_{X_r}) \rvert^2 \Big) \Big) \Delta t,  \label{numdx}\\
&\max_{k \in \{1, 2, \dots, N\}} \mathbb{E}\Big( \lvert Y^x_{t_k} - Y_{t_k}^{n,x} \rvert^2 \Big)  
\leq 2C \Delta t.  \notag
\end{align}
In the same way we show \eqref{equa1}, we can see that $\vert Y_t^{n,x} \vert^{-1}$ belongs to $L^p(P)$, for every $p\geq 2$. Here, we show the convergence of $\frac{\partial}{\partial{x}}{X_t ^{n,x}}$ to $\frac{\partial}{\partial{x}}{X_t ^x}$. 
\begin{lemma}\label{lempartialxxn}
Under Assumption \ref{assum1} and \ref{assume2}, for every $p\geq 2$ we have
$$\lim_{n \rightarrow \infty} \mathbb{E}\Big( \sup_{0 \leq t \leq T}\lvert \frac{\partial}{\partial{x}}{X_t ^x} -\frac{\partial}{\partial{x}}{X_t ^{n,x}}  \rvert^p \Big) =0 .$$
\end{lemma}
\begin{proof}
We first know that 
\begin{align}
\frac{\partial}{\partial{x}}{X_t ^x} -\frac{\partial}{\partial{x}}{X_t ^{n,x}} &= \int_{0}^{t}\Big( {A_s ^x}\frac{\partial}{\partial x}X_s ^x -A_s ^{n,x}\frac{\partial}{\partial x}X_s ^{n,x}+ \alpha _s ^x-\alpha _s ^{n,x}\Big)ds \notag\\
&+ \sum_{m=1}^{d}\int_{0}^{t}\Big( C_s(\frac{\partial}{\partial x}X_s ^x-\frac{\partial}{\partial x}X_s ^{n,x}) + \beta _s ^{x,m}- \beta _s ^{n,x,m}\Big)d{W_s ^n}\notag\\
&+ \int_{0}^{t}\int_{\mathbb{R}_0}^{}\Big( {{F_{s,z}}(\frac{\partial}{\partial x}X_s ^x-\frac{\partial}{\partial x}X_s ^{n,x})+ \gamma _s ^x-\gamma _s ^{n,x}}\Big)\tilde{N}(dz, ds), \label{partialxxn}
\end{align}
where
\begin{align*}
\alpha _s ^{n,x}& := \partial_1 b(s, \rho_{X_s ^{n,x}}){X_s ^{n,x}}\frac{\partial}{\partial x}\rho_{X_s^{n,x}},  \qquad 
\beta _t ^{n,x,m}:= \partial_2{\sigma_n}(t, X_t ^{n,x}, \pi_{X_t ^x})\frac{\partial}{\partial x}\pi_{X_t ^x},\\
\gamma _t ^{n,x}& := \partial_3{\lambda}(t, X_t ^{n,x}, z, \eta_{X_t ^{n,x}})\frac{\partial}{\partial x}\eta_{X_t ^{n,x}} . 
\end{align*}
In accordance with Assumption \ref{assume2}, there exists some another constant $C$ that 
 \begin{align}
& \vert \alpha _s ^{n,x}-\alpha _s ^{x}\vert^p\notag\\
 &\leq  3^{p-1}\Big(\vert  \partial_1 b(s, \rho_{X_s ^{n,x}}) \vert^p \vert X_t ^{n,x}-X_t ^{x} \vert^p+\vert  \partial_1 b(s, \rho_{X_s ^{n,x}})- \partial_1 b(s, \rho_{X_s ^{x}})\vert^p \notag\\
 & \hspace{2cm}\vert X_t ^x \vert^p\vert \frac{\partial}{\partial x}\rho_{X_s^{x}}\vert^p +\vert \frac{\partial}{\partial x}\rho_{X_s^{n,x}}-\frac{\partial}{\partial x}\rho_{X_s^{x}}\vert^p \vert \partial_1 b(s, \rho_{X_s ^{n,x}}){X_s ^{n,x}}\vert^p \Big)\notag\\
& \leq  3^{p-1}C \Big(\vert X_t ^{n,x}-X_t ^{x} \vert^p +\mathbb{E} \Big(\vert  \frac{\partial}{\partial{x}}{X_t ^x} -\frac{\partial}{\partial{x}}{X_t ^{n,x}}\vert^p\Big)\notag \\
&+ 3^{p-1}C\vert X_t ^{n,x}\vert^p \mathbb{E}\Big(\vert  \frac{\partial}{\partial{x}}{X_t ^x} -\frac{\partial}{\partial{x}}{X_t ^{n,x}}\vert^p\Big)\Big),\label{alphaxxn}
\end{align}
and similarly,
\begin{align}
& \vert \beta _s ^{x,m}- \beta _s ^{n,x,m}\vert^p \leq C2^{p-1}\Big( \mathbb{E}\Big(\vert  \frac{\partial}{\partial{x}}{X_t ^x} -\frac{\partial}{\partial{x}}{X_t ^{n,x}}\vert^p\Big)+ \mathbb{E}\Big(\vert  X_t ^{n,x} -{X_t ^{n,x}}\vert^p\Big)\Big), \label{betaxxn}
\end{align}
\begin{align}
 \vert \gamma _s ^{x}- \gamma _s ^{n,x}\vert^p \leq C2^{p-1}\Big( \mathbb{E}\Big(\vert  \frac{\partial}{\partial{x}}{X_t ^x} -\frac{\partial}{\partial{x}}{X_t ^{n,x}}\vert^p\Big)+ \mathbb{E}\Big(\vert  X_t ^{n,x} -{X_t ^{n,x}}\vert^p\Big)\Big). \label{gammaxxn}
\end{align}
Now, taking expectation of the equation \eqref{partialxxn} to the power $p$, utilize Burkholder-Davis-Gundy inequality to derive 
\begin{align*}
\mathbb{E}\Big(&\sup_{0 \le t \leq T}\vert \frac{\partial}{\partial{x}}{X_t ^x} -\frac{\partial}{\partial{x}}{X_t ^{n,x}}\vert^p\Big) \\
&\leq C\int_{0}^{T}\mathbb{E}\Big( \vert{A_s ^x}\frac{\partial}{\partial x}X_s ^x -A_s ^{n,x}\frac{\partial}{\partial x}X_s ^{n,x}\vert^p\Big)ds + C\int_{0}^{T}\mathbb{E}\Big(\vert \alpha _s ^x-\alpha _s ^{n,x}\vert^p\Big)ds \\
&+CT\sum_{m=1}^{d} \int_{0}^{T}\mathbb{E}\Big(\vert\beta _s ^{x,m}- \beta _s ^{n,x,m}\vert^p\Big)ds+ CdT^p \int_{0}^{T}\mathbb{E}\Big(\vert C_s(\frac{\partial}{\partial x}X_s ^x-\frac{\partial}{\partial x}X_s ^{n,x}\vert^p\Big)ds \\
&+ C\int_{0}^{T}\int_{\mathbb{R}_0}^{}\mathbb{E}\Big(\vert F_{s,z}(\frac{\partial}{\partial x}X_s ^x-\frac{\partial}{\partial x}X_s ^{n,x})\vert^p\Big)\mu(dz)ds\\
&+C\int_{0}^{T}\int_{\mathbb{R}_0}^{}\mathbb{E}\Big(\vert \gamma _s ^x-\gamma _s ^{n,x}\vert^p\Big)\mu(dz)ds,
\end{align*}
Substitute the bound \eqref{alphaxxn}, \eqref{betaxxn} and \eqref{gammaxxn} into the above inequality and then apply lipschitz condition of $b$ to result ($C$ is a constant can be different each time)
\begin{align*}
\mathbb{E}\Big(&\sup_{0 \le t \leq T}\vert \frac{\partial}{\partial{x}}{X_t ^x} -\frac{\partial}{\partial{x}}{X_t ^{n,x}}\vert^p\Big) \\
&\leq C\int_{0}^{T}\mathbb{E}\Big( \vert{A_s ^x}\frac{\partial}{\partial x}X_s ^x -A_s ^{n,x}\frac{\partial}{\partial x}X_s ^{n,x}\vert^p\Big)ds + C\int_{0}^{T} \mathbb{E}\Big(\vert  X_t ^{n,x} - {X_t ^{n,x}}\vert^p\Big)ds \\
&+ C\int_{0}^{T}( C_s^p+\int_{\mathbb{R}_0}F^p_{s,z}\mu(dz)+3^{p+1})\mathbb{E}\Big(\vert\frac{\partial}{\partial x}X_s ^x-\frac{\partial}{\partial x}X_s ^{n,x}\vert^p\Big)ds \\
&\leq C\int_{0}^{T} \mathbb{E}\Big(\vert  X_t ^{n,x} - {X_t ^{n,x}}\vert^p\Big)\mathbb{E}\Big( \vert\frac{\partial}{\partial x}X_s ^{x}\vert^p\Big)ds + C\int_{0}^{T} \mathbb{E}\Big(\vert  X_t ^{n,x} - {X_t ^{n,x}}\vert^p\Big)ds \\
&+ C\int_{0}^{T}\Big((A_s^x)^p+C_s^p+\int_{\mathbb{R}_0}F^p_{s,z}\mu(dz)+3^{p+1}\Big)\mathbb{E}\Big(\vert\frac{\partial}{\partial x}X_s ^x-\frac{\partial}{\partial x}X_s ^{n,x}\vert^p\Big)ds. \\
\end{align*}
The Gronwall inequality and Theorem \ref{thmthm}, conclude the result and complete the proof.
\end{proof}
\noindent Another significant topic for consideration is the convergence of Delta through the implementation of the Euler method for two pay-off functions, Call options and Lookback barrier options. To this end, define
\begin{equation*}
	{\vartheta^n}(x)=\mathbb{E}\Big({\vartheta}_{X^n,G^n}(x)\Big):= {\mathbb{E}}\Big({\Phi}(X_T^{n,x}, G^n(X^{n,x})) \Big).
\end{equation*} 
To validate convergence, one must provide evidence that $\Delta^n \xrightarrow{n \to \infty} \Delta.$
Given the fundamental differences in the payoff structures of European call options and barrier options, it is necessary to establish the convergence of $\Delta$ through separate, tailored proofs for each type of option. 
 \subsection{European call option}
 \noindent We begin by demonstrating the convergence for European call options. We are inspired by the proof of Theorem 5.1 from \cite{higham2005convergence}. Hence, we note that for every $\epsilon$ there exists a $\delta_0$ such that $$P(K-\delta_0 < X_T^{x}< K) < \frac{\epsilon}{2}, \qquad P(K < X_T^{x}< K+\delta_0) < \frac{\epsilon}{2}.$$
 Define $$\mathcal{B}_1:=\{X_T^{x}> K\}, \qquad \mathcal{B}_2:=\{ X_T^{n,x}>K\}.$$
 \begin{align}\label{bb}
 P(\mathcal{B}_2 \cap \mathcal{B}_1^c) 
 & \leq P\Big(\vert X_T^{x}- X_T^{n,x}\vert > \delta_0)+P(\{K-\delta_0 < X_T^{x}< K \} \cap \{ X_T^{n,x}> K\}\Big) \notag\\
  & \leq C\frac{1}{\delta_0^2}(1+\vert x\vert^2) (\Delta t)+ \epsilon,
 \end{align} 
 \begin{align}\label{bbc}
P(\mathcal{B}_1 \cap \mathcal{B}_2^c) 
 & \leq P\Big(\vert X_T^{x}- X_T^{n,x}\vert > \delta_0)+P(\{K< X_T^{x}< K+\delta_0 \} \cap  \{X_T^{n,x}< K\}\Big)\notag\\
  & \leq C\frac{1}{\delta_0^2}(1+\vert x\vert^2) (\Delta t)+ \epsilon.
 \end{align} 
 From \eqref{bb}, \eqref{bbc} and Lemma \ref{lempartialxxn} we have 
\begin{align*}
\lim_{n \rightarrow \infty} \vert \Delta^n - \Delta\vert^2 & = \lim_{n \rightarrow \infty}\mathbb{E} \Big(1_{X_T^{x}>K} \frac{\partial}{\partial x} X_T^x  -  1_{X_T^{n,x}>K} \frac{\partial}{\partial x} X_T^{n,x} \Big)^2  \\& \leq \lim_{n \rightarrow \infty} \mathbb{E}\Big( \mathbf{1}_{\mathcal{B}_1^c \cap \mathcal{B}_2}\frac{\partial}{\partial x}X_T^x\Big)^2 +\lim_{n \rightarrow \infty}\mathbb{E}\Big(1_{\mathcal{B}_1 \cap \mathcal{B}_2^c }\frac{\partial}{\partial x} X_T^x \Big)^2 \\
&+   \lim_{n \rightarrow \infty}\mathbb{E} \Big(1_{X_T^{n,x}>K} (\frac{\partial}{\partial x} X_T^x -\frac{\partial}{\partial x} X_T^{n,x})^2 \Big)=0.\\
 \end{align*}
\subsection{Barrier Options} 
\noindent We broaden our analysis to encompass barrier options, demonstrating the convergence of the Euler scheme for $\Delta$ in the up-and-out call option. A similar procedure in European options can be employed to derive the conclusion.
To achieve this, we note that for every $\epsilon>0$ there exists some constant $\delta_0 >0$ such that 
$$P(\mathcal{B}_6):= P(\{\vert  G_M(X^{x})- K \vert < \delta_0\})+P(\{\vert  G_M(X^{x})- B \vert < \delta_0 \}) < \epsilon,$$
 and define the sets 
$$\mathcal{B}_5:= \{1_{K < G_M(X^{x})< B}\},  \qquad \mathcal{B}^n_5:= \{1_{K < G_M(X^{n,x})< B}\},$$
$$\mathcal{B}_6:=\{K < G_M(X^{x})< K+ \delta_0\} \cup \{B -\delta_0 < G_M(X^{x})< B\},$$
$$\mathcal{B}_7:=\{ K- \delta_0 < G_M(X^{x})< K\} \cup \{B < G_M(X^{x}) < B +\delta_0\}.$$
We can derive
\begin{align}
&\vert \Delta -\Delta^n \vert^2 \notag\\
&=\mathbb{E} \Big(\vert 1_{\mathcal{B}_5 \cap \mathcal{B}^n_5} \vert \frac{\partial}{\partial x}(G_M(X^{n,x}))-\frac{\partial}{\partial x}(G_M(X^{x}))\vert^2 \Big)\notag\\
& +  \mathbb{E}\Big(1_{\mathcal{B}_5 \cap (\mathcal{B}^n_5 )^c}\Big) \mathbb{E}\Big(\vert \frac{\partial}{\partial x}(G_M(X^{x}))\vert^2\Big)+  \mathbb{E}\Big(1_{\mathcal{B}_5^c \cap \mathcal{B}^n_5 }\Big) \mathbb{E}\Big(\vert \frac{\partial}{\partial x}(G_M(X^{n,x}))\vert^2\Big)\notag\\
& \leq \mathbb{E} \Big( 1_{\mathcal{B}_5 \cap \mathcal{B}^n_5} \sup_{0 \leq t \leq T}\vert \frac{\partial}{\partial x}X_t^{n,x}-\frac{\partial}{\partial x}X_t^{x}\vert^2 \Big)\notag\\
& +  \mathbb{E}\Big(1_{\mathcal{B}_5 \cap (\mathcal{B}^n_5 )^c}\Big) \mathbb{E}\Big(\vert \frac{\partial}{\partial x}(G_M(X^{x}))\vert^2\Big)+  \mathbb{E}\Big(1_{\mathcal{B}_5^c \cap \mathcal{B}^n_5 }\Big) \mathbb{E}\Big(\vert \frac{\partial}{\partial x}(G_M(X^{n,x}))\vert^2\Big), \label{eltaelta}
\end{align}
We will show that for enough large $n$, the events $\mathcal{B}_5 \cap (\mathcal{B}^n_5 )^c$ or $\mathcal{B}_5^c \cap \mathcal{B}^n_5$ does not happen. We proceed with our discussion of European options to result 
\begin{align*}
P( \mathcal{B}_5 \cap (\mathcal{B}^n_5)^c) &= P(\{K+\delta_0 < G_M(X^x) < B-\delta_0\} \cap (\mathcal{B}^n_5)^c)+P(\mathcal{B}_6 \cap (\mathcal{B}^n_5)^c)\\
&\leq P(\{\vert \sup_{0 \leq t \le T} X_t^x - \sup_{0 \leq t \le T} X_t^{n,x}\vert > \delta_0\} )+\epsilon\\
 &\leq \frac{1}{\delta_0^2}\mathbb{E}( \sup_{0 \leq t \le T} \vert X_t^x - X_t^{n,x}\vert^2)+\epsilon.
\end{align*}
Also, in a similar course,
\begin{align*}
P( \mathcal{B}_5^c \cap \mathcal{B}^n_5) &= P([\{ G_M(X^x) < K-\delta_0\} \cup  \{B+\delta_0 < G_M(X^x) \} ]\cap \mathcal{B}^n_5)\\
&+P(\{\mathcal{B}_7 \cap \mathcal{B}^n_5\}\\
&\leq P(\{\vert \sup_{0 \leq t \le T} X_t^x - \sup_{0 \leq t \le T} X_t^{n,x}\vert > \delta_0\} )+\epsilon\\
 &\leq \frac{1}{\delta_0^2}\mathbb{E}( \sup_{0 \leq t \le T} \vert X_t^x - X_t^{n,x}\vert^2)+\epsilon.
\end{align*}
Hence, substituting above inequalities and Corollary \ref{coro} into \eqref{eltaelta} yields there exists some positive constant $C$ such that 
\begin{align*}
\vert \Delta -\Delta^n & \vert^2\leq \mathbb{E} \Big(  \sup_{0 \leq t \le T} \vert \frac{\partial}{\partial x} X_t^x - \frac{\partial}{\partial x} X_t^{n,x}\vert^2 \Big) \\
&+ \Big(\frac{1}{\delta_0^2}\mathbb{E}( \sup_{0 \leq t \le T} \vert X_t^x - X_t^{n,x}\vert^2)+\epsilon\Big) \mathbb{E}\Big(\vert \sup_{0 \leq t \leq T}\frac{\partial}{\partial x}X_t^{n,x}\vert^2 +\vert \sup_{0 \leq t \leq T}\frac{\partial}{\partial x}X_t^{x}\vert^2\Big).
\end{align*}
Lemma \ref{lempartialxxn}, the equation \eqref{Eulerorder}, Corollary \ref{coro} and Theorem \ref{thmthm} estate that 
$$\lim_{n \rightarrow \infty} \vert \Delta -\Delta^n \vert =0.$$

\section{Numerical experiments}
\noindent To compute the Malliavin weights, we utilized a numerical method grounded in the relations derived in the previous section, specifically employing the Euler scheme due to its favorable convergence properties. The computation required simulations of \(X_t^x\), \(Y_t\), \(u(t)\), and \(D_{t,z}X_t^x\), which were performed with a simulation size of \(M = 1000\) to ensure convergence and validate the accuracy of the results. However, as illustrated in the figures, an acceptable level of mean-square error approximation was achieved with a reduced simulation size of \(M = 500\), demonstrating its efficiency. We adopted a uniform time partition, where the number of partitions \(N\) is given by \(N = \frac{T}{\Delta t}\), and small step sizes of \(\Delta t = 2^{-12}\) to align the numerical solution with the mean-field SDE with jumps.  
\noindent The primary computational difficulty arose in the context of barrier options, particularly due to the complexity in computing \(\frac{\partial}{\partial x} \max_{t} X_t^x\). Unlike European options, barrier options involve path-dependent features, significantly increasing the computational burden. To address these challenges, we explain the computational steps involved in this process in the following examples.  
\noindent We focus on a European call option and an up-and-out barrier option to illustrate the efficiency and applicability of our results in a numerical context. The European call option is a benchmark due to its simplicity and well-understood properties, while the up-and-out barrier option highlights the challenges associated with path dependence. For the European call option, after computing the Malliavin weights as derived in (\ref{wEur}), we employ the Monte Carlo simulation method to calculate the delta, as described in (\ref{Delta}). This approach allows us to validate the accuracy and convergence of numerical schemes in a straightforward setting. In the case of barrier options, the computation of Malliavin weights requires a careful evaluation of $\frac{\partial}{\partial x} \max_{t} X_t^x$, as required by Formula (\ref{W2B}).  Moreover, it is crucial to ensure that the theoretical conditions outlined in the article are satisfied within our numerical examples, as this verification step is essential to guarantee the accuracy, validity, and robustness of the obtained results. 
To demonstrate the effectiveness of our results, we compare it to the following finite difference method scheme hereunder  
\begin{align*}  
\vartheta'(x) \approx \frac{\mathbb{E}[\Phi(X_T^{x+h}, G(X^{x+h}))] - \mathbb{E}[\Phi(X_T^x, G(X^x))]}{h}, \quad h \approx 0.  
\end{align*} 
\begin{example}\label{E1}
\noindent  Consider the following mean-field SDE with jump 
\begin{align}\label{example1}
&d{X_t ^x}={\frac{a(\mathbb{E}({X_t ^x}) + 1) }{\mathbb{E} ({X_t ^x})}}X_t ^xdt + b{X_t ^x}d{W_t} + {\int_{\mathbb{R}_0}^{}{c {X_t ^x}}} {\tilde{N}}(dz,dt), \notag\\
&X_0 = x_0,
\end{align}
\begin{figure}[H]
\centering
    \includegraphics[width=14cm]{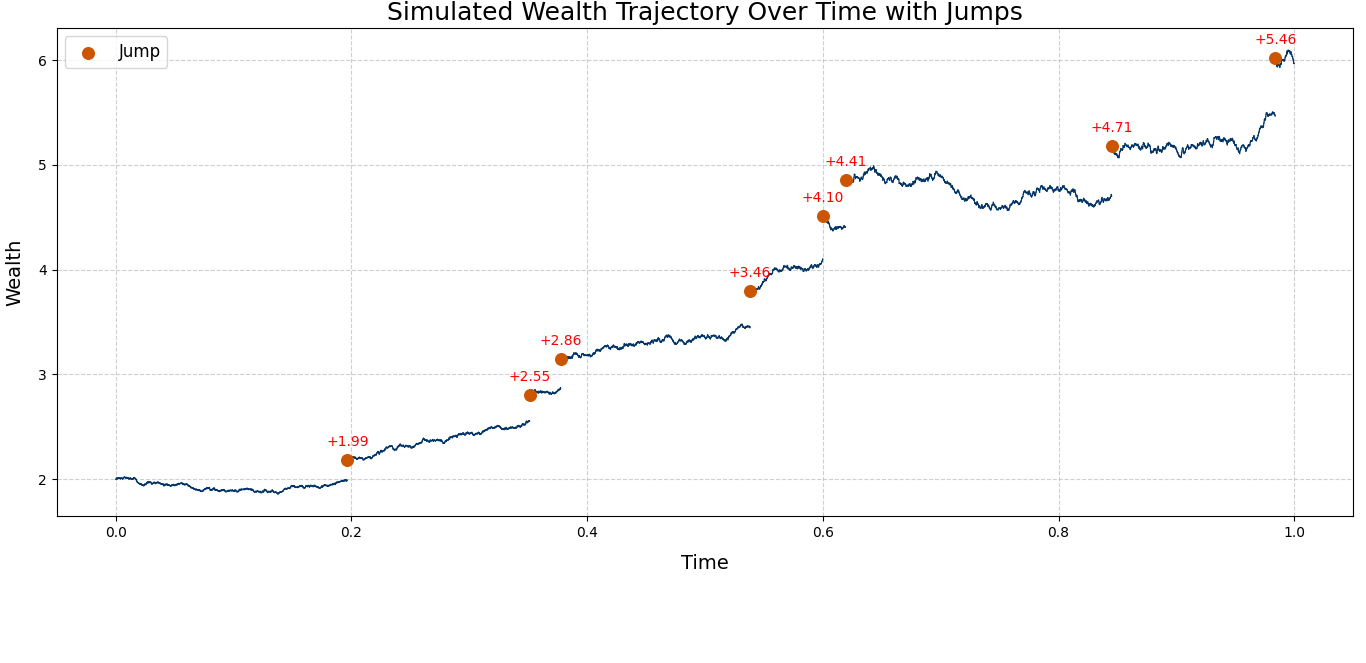}
    \caption{\scriptsize On top parameters set to $a= b= c= 1$, $\lambda = 0.1$, $T=1$, Example \ref{E1}.}\label{Figure1}
\end{figure}
\noindent Where \(a\), \(b\), and \(c\) are constants, which, for consistency, are assumed to satisfy \(a = b = c = 1\).   Assume that the jump sizes $\{Z_i, i=1, ..., N_T\}$ have the uniform distribution on $[\frac{-1}{2}, \frac{1}{2}$]. Also jumps occur on times by exponential and gamma distribution with intensity $\lambda$. A simulation of (\ref{example1}) is obtained in Figure (\ref{Figure1}) by using the Euler method with a step size of $2 ^{-24}$.  Following the drift coefficient, we derive that
\begin{align*}  
b(t, f(t)) &= \frac{0.2 (1 + f(t))}{f(t)},
\end{align*}
\noindent By performing calculations as outlined in (\ref{ft}), we arrive at the following expression
\begin{align*}
f(t) &= {(1+x)} e^{5(t+1)} - 1,
\end{align*}
\noindent consequently,
\begin{align*}
b(t, f(t)) &= \frac{1}{5} \frac{{(1+x)} e^{5(t+1)}}{-1 + {(1+x)} e^{5(t+1)}}.
\end{align*} 
\noindent In light of this, 
\begin{align*}
X_t ^{x} =
& = x\exp\Big\{\frac{1}{5}{\int_{0}^{t}{\frac{{(1+x)} e^{5(t+1)}}{-1 + {(1+x)} e^{5(t+1)}}}}\Big\} \mathcal{K}_t u(0),
\end{align*}
\noindent where $u(0) = 1,$ and here, \( \mathcal{K}_t \) is defined as
\begin{align*}  
\mathcal{K}_t &= \exp(H(t)) = \exp\left(-\frac{t}{2} + \lambda t [\ln(2) - 1] + W_t + \int_{0}^{t} \int_{\mathbb{R}_0} \ln(2) \, \tilde{N}(dz, ds)\right),  
\end{align*}  
\noindent where \( H(t) \) encapsulates the stochastic components of the system, including the Wiener process \( W_t \) and the compensated Poisson random measure \( \tilde{N}(ds, dz) \).  \\
\noindent To analyze the system further, we maximize the integral of \( b(s, f(s)) \) over the interval \([0, T]\), leading to:  

\begin{align*}  
\max_{t} \int_{0}^{t} b(s, f(s)) \, ds &=  \frac{1}{25}\max_{t \in [0,1]} \ln\left(\frac{-1 + {(1+x)e^{5(t+1)}}}{-1 +{(1+x)e^5}}\right).  
\end{align*}  

\noindent This result provides a key insight into the system's behavior over time, particularly in terms of the growth dynamics governed by \( b(t, f(t)) \).  

%
\noindent Substituting the expressions for \( b(t, f(t)) \) and \( \mathcal{K}_t \), we derive
\begin{align*}  
\frac{\partial}{\partial x} \max_{t} X_t^x &= \frac{1}{25} \Big( \max_t \mathcal{K}_t \Big) \cdot \frac{\partial}{\partial x} \Bigg( x \cdot \frac{-1 + {(1+x)} e^{5(t+1)}}{-1 +{(1+x) e^{5}}}\Bigg)\\
&= \big({\max_t \mathcal{K}_t }\big)\Big( \Big(\frac{-1 +{(1+x)} e^{10}}{-1 +{(1+x) e^{5}}}\Big)^{\frac{1}{25}}+ \frac{x}{5} \Big(\frac{-1 +{(1+x)} e^{10}}{-1 +{(1+x) e^{5}}}\Big)^{\frac{-24}{25}} \notag\\
&\cdot\Big( \frac{e^{10}(-1 + e^{5}{(1+x)}) - e^{5}(-1 + e^{10}{(1+x)})}{(-1 + e^{5}{(1+x)})^2}\Big)\Big).
\end{align*}  
\noindent Ultimately, substituting the initial value $x = 1$ into the expression yields the following result
\begin{align*}  
\frac{\partial}{\partial x} \max_{t} X_t^x = &= \Big(1 + \frac{1}{5} \cdot \frac{e^{10} - e^5}{-1 + e^{10}}\Big) \cdot \Big(\max_t \mathcal{K}_t\Big).  
\end{align*}  
\noindent The exclusive topic that we must address for greater clarity is
\begin{align*}
\mathcal{K}_n \xrightarrow{n \to \infty} \mathcal{K}_t,
\end{align*}
\noindent where $\mathcal{K}_n$ is shown the Euler scheme approximation of $\mathcal{K}_t.$ By It\^o Formula 
for $f(x) = \exp(x)$ and
\begin{align*}
  dH(t) = \big((-\frac{1}{2} + \lambda [\ln(2) - 1]\big) dt + dW_t + \int_{\mathbb{R}_0} \ln(2) \, \tilde{N}(dt, dz),  
\end{align*}
the final expression for the differential of \( \mathcal{K}_t \) is:  
\begin{align*} 
d\mathcal{K}_t = \mathcal{K}_t \left( \lambda [\ln(2) - 1] dt + dW_t + \int_{\mathbb{R}_0} \ln(2) \, \tilde{N}(dt, dz) \right).  
\end{align*}  
\noindent The convergence of the Euler scheme for $\mathcal{K}_t$ is detailed in \cite{platen2010numerical} specifically in Theorem (6.4.1).

\begin{figure}[H]
\centering
             \subfloat{\includegraphics[width=14cm]{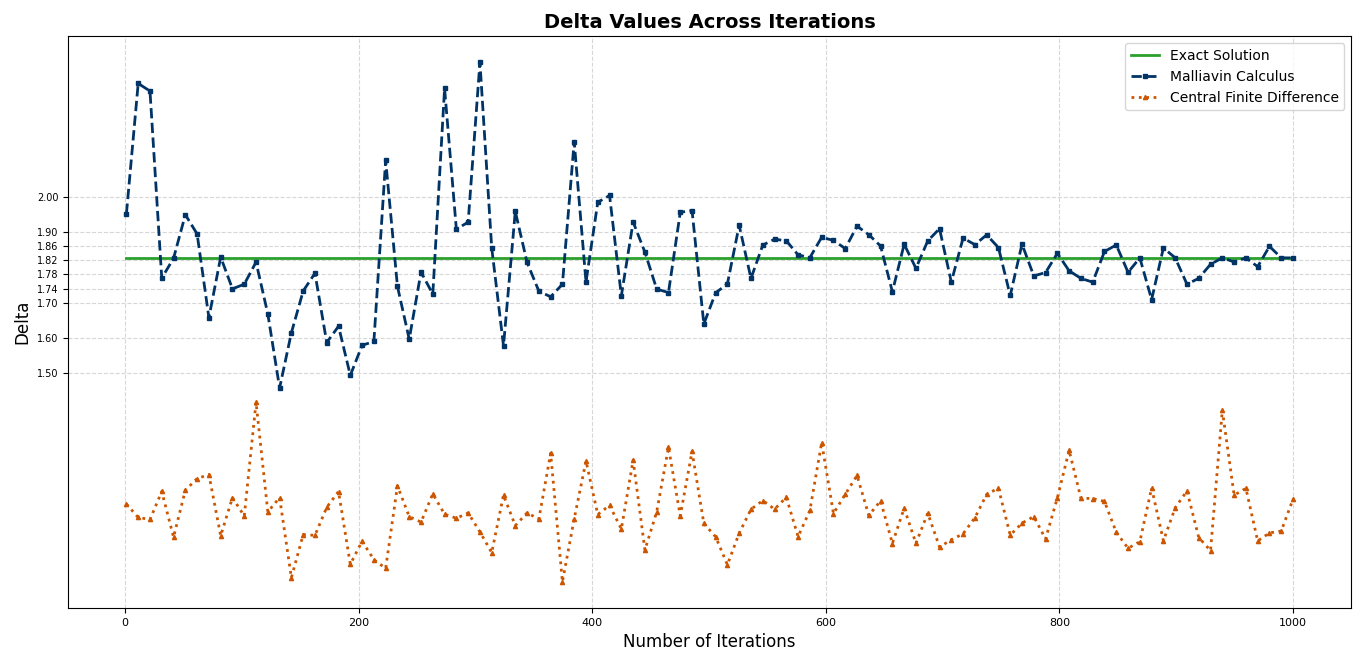}}
             \caption{\scriptsize Comparison of Malliavin calculus with Central Finite Difference method for European Call Options. On top parameters set to $a= b= c= 1$, $T=1$ and h=0.001 for Central Finite Difference, Example \ref{E1}.}\label{fig2-EF}
\end{figure}
\noindent To ensure that the examples align with the conditions established in the article, it is instructive to consider the following relationship
\begin{align*}
D_{r,z}\Big({e^{\max_{0 \leq t \leq T} H(t)}}\Big) =
{e^{\max_{t} H(t)}} \Big({{e^{\ln2}} - 1}\Big)= {e^{\max_{t} H(t)}}.
\end{align*}

\begin{figure}[H]
\centering
             \subfloat{\includegraphics[width=14cm]{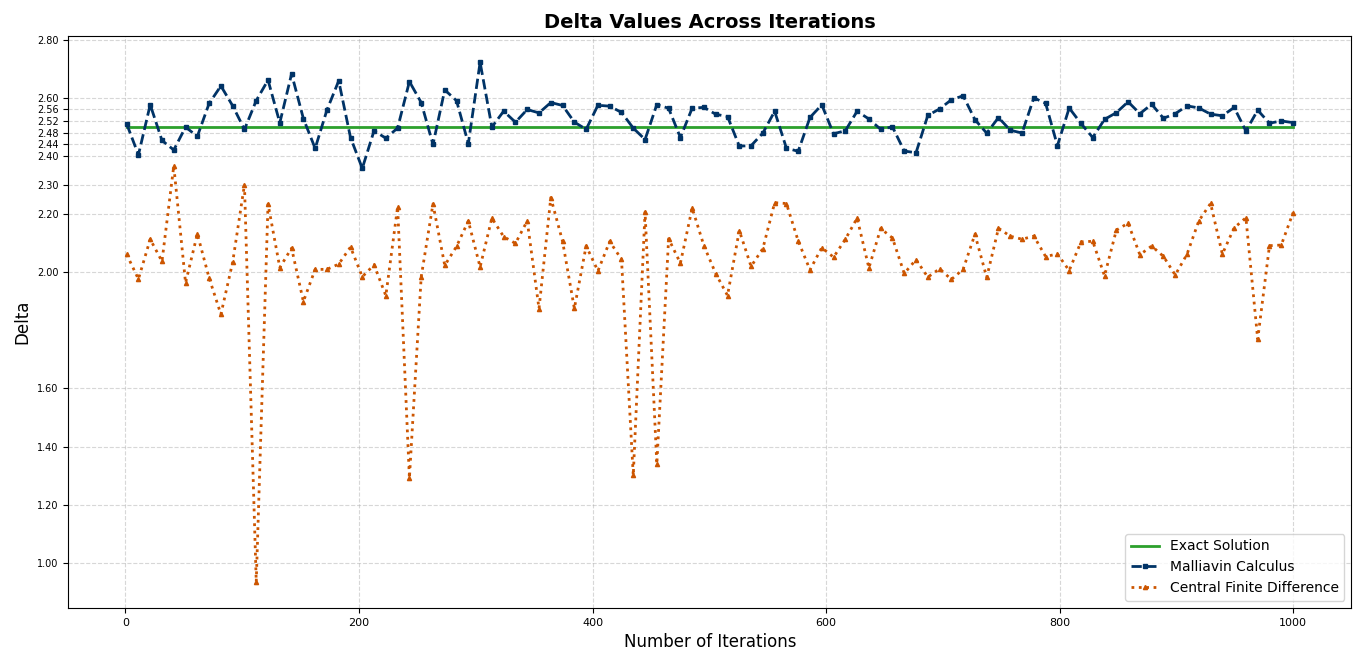}}
             \caption{\scriptsize Comparison of Malliavin calculus with Central Finite Difference method for Barrier Up and Out Options. On top parameters set to $a= b = c= 1$, $T=1$ and h=0.001 for Central Finite Difference, Example \ref{E1}.}\label{fig2-EF}
\end{figure}
\begin{figure}[H]
\centering
             \subfloat{\includegraphics[width=14cm]{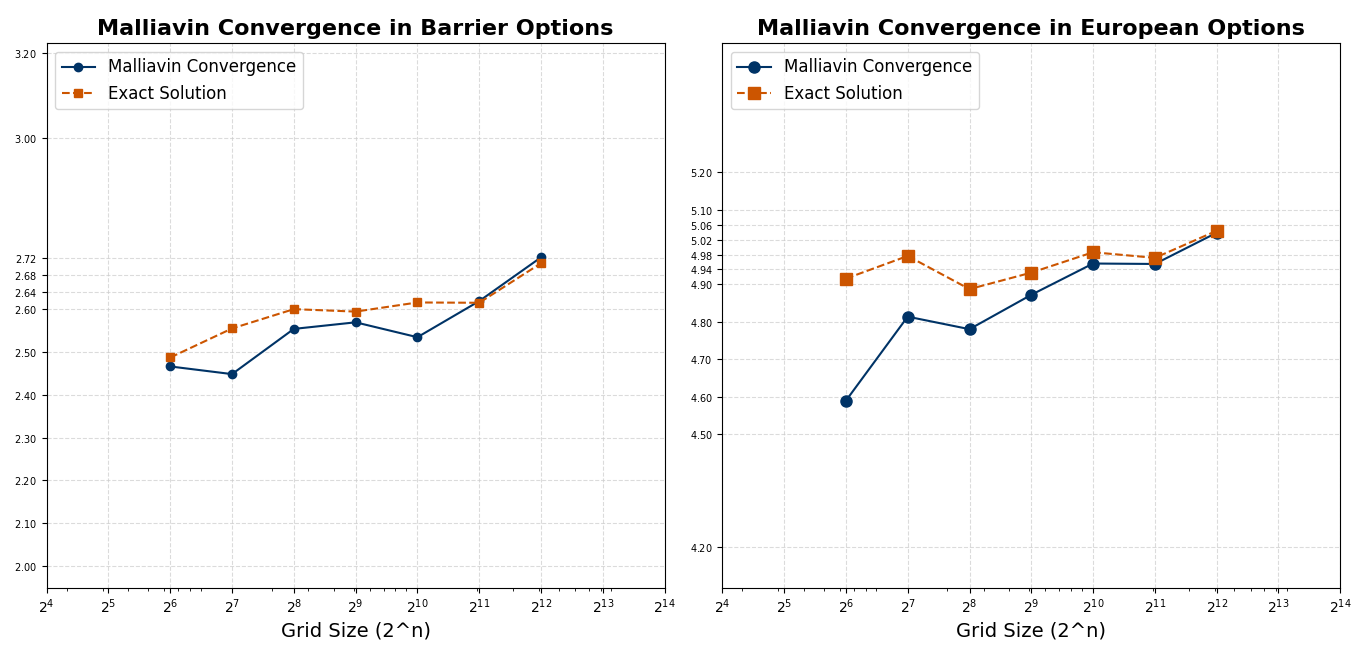}}
             \caption{\scriptsize Comparison Convergence rate of Malliavin Method with Exact solution,  Example \ref{E1}}\label{fig2-EF}
\end{figure}

\begin{example}\label{E2}
\noindent Consider the following mean-field SDE with jump
\begin{align*}  
dX_t^x &= a \mathbb{E}(X_t^x) X_t^x \, dt   
       + b X_t^x \, dW_t   
       + \int_{\mathbb{R}_0} c X_t^x \, \tilde{N}(dz, dt), \notag \\
X_0 &= x_0. 
\end{align*}
\end{example}
\begin{figure}[H]
\centering
    \includegraphics[width=14cm]{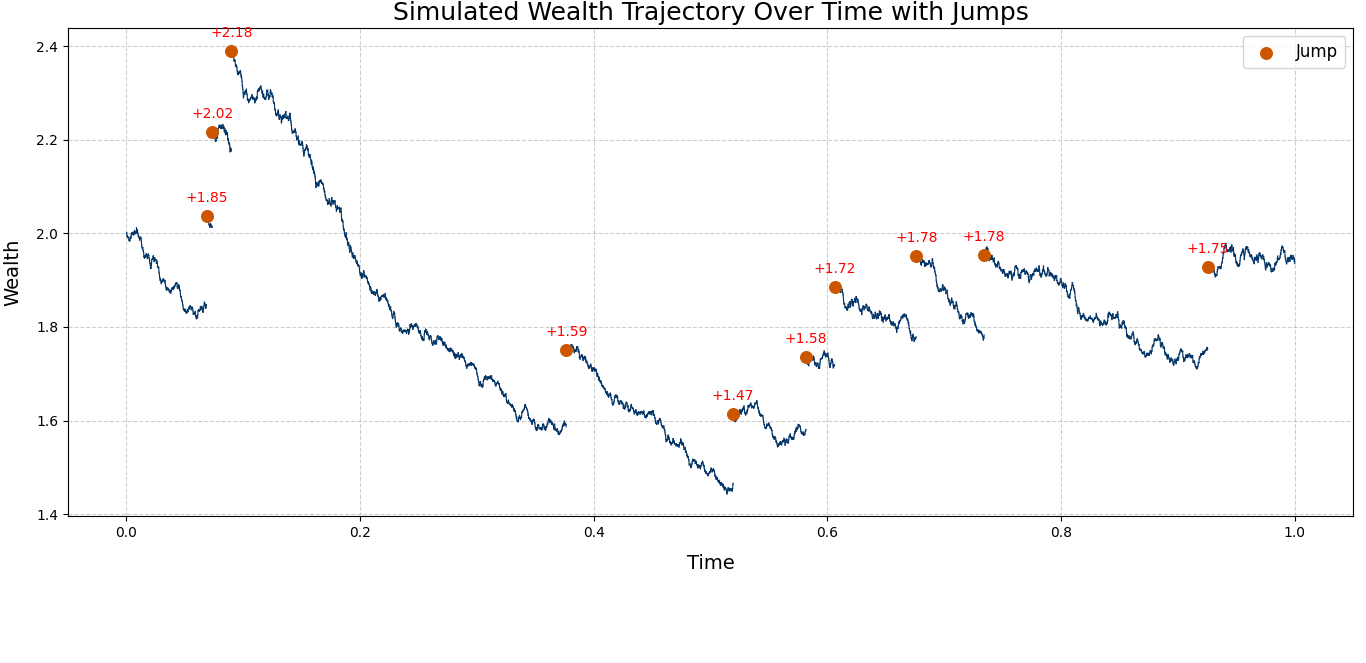}
    \caption{\scriptsize On top parameters set to $a= -1, b=c=1$, $T=1$, Example \ref{E2}.}\label{Figure2}
\end{figure}
\noindent Based on the drift coefficient, we can ascertain that $b(t, f(t)) = -{f(t)}$ and through the calculations in (\ref{ft}), we derive the expression
\begin{align*}
f(t) &= \frac{x}{xt+1}.
\end{align*}
\noindent We maximize the integral of \( b(s, f(s)) \) over the interval \([0, T]\), yielding 

\begin{align*}  
\exp\Big(\max_{t} \int_{0}^{T} b(s, f(s)) \, ds \Big)= \exp\Big(\max_{t \in [0,1]}(\ln\frac{1}{xt+1})\Big) = 1, 
\end{align*} 
\noindent therefore, we achieve
\begin{align*}
\frac{\partial}{\partial x} \max_{t} X_t^x = \frac{\partial}{\partial x} \Big({xu(0)\max_{t \in [0,1]} S_t}\Big) = 
= {\max_t \mathcal{K}_t }.
\end{align*}

\end{example}
\begin{figure}[H]
\centering
             \subfloat{\includegraphics[width=14cm]{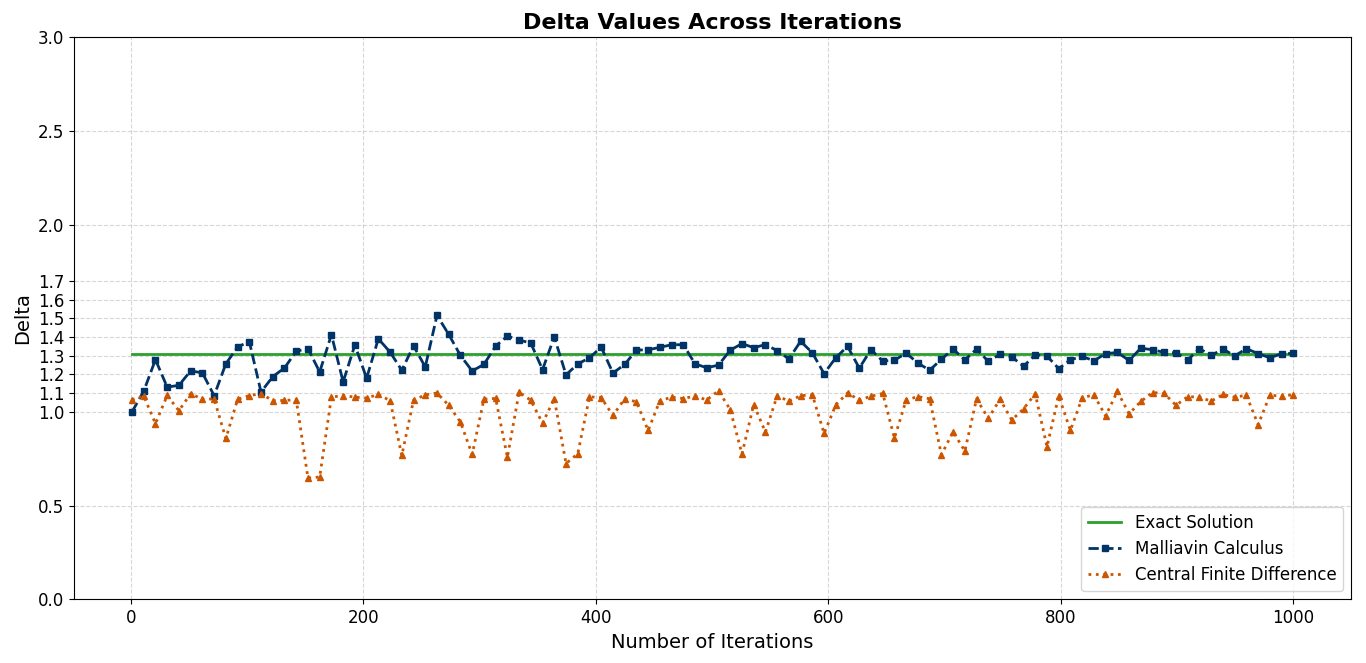}}
             \caption{\scriptsize Comparison of Malliavin calculus with Central Finite Difference method for European Call Options. On top parameters set to $a=-1,  b= c= 1$, $T=1$ and h=0.001 for Central Finite Difference, Example \ref{E2}.}\label{fig2-EF}
\end{figure}
\begin{figure}[H]
\centering
             \subfloat{\includegraphics[width=14cm]{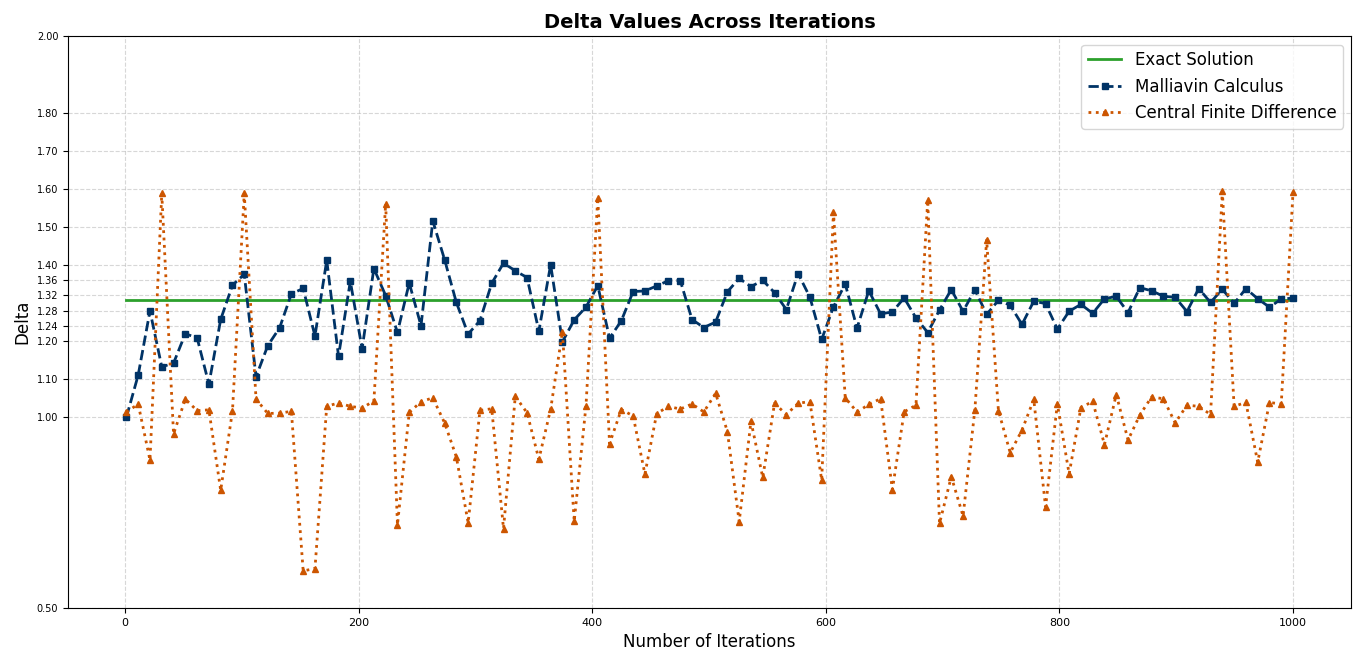}}
             \caption{\scriptsize Comparison of Malliavin calculus with Central Finite Difference method for Barrier Up and Out Options. On top parameters set to $a= -1, b = c= 1$, $T=1$ and h=0.001 for Central Finite Difference, Example \ref{E2}.}\label{fig2-EF}
\end{figure}
\begin{figure}[H]
\centering
             \subfloat{\includegraphics[width=14cm]{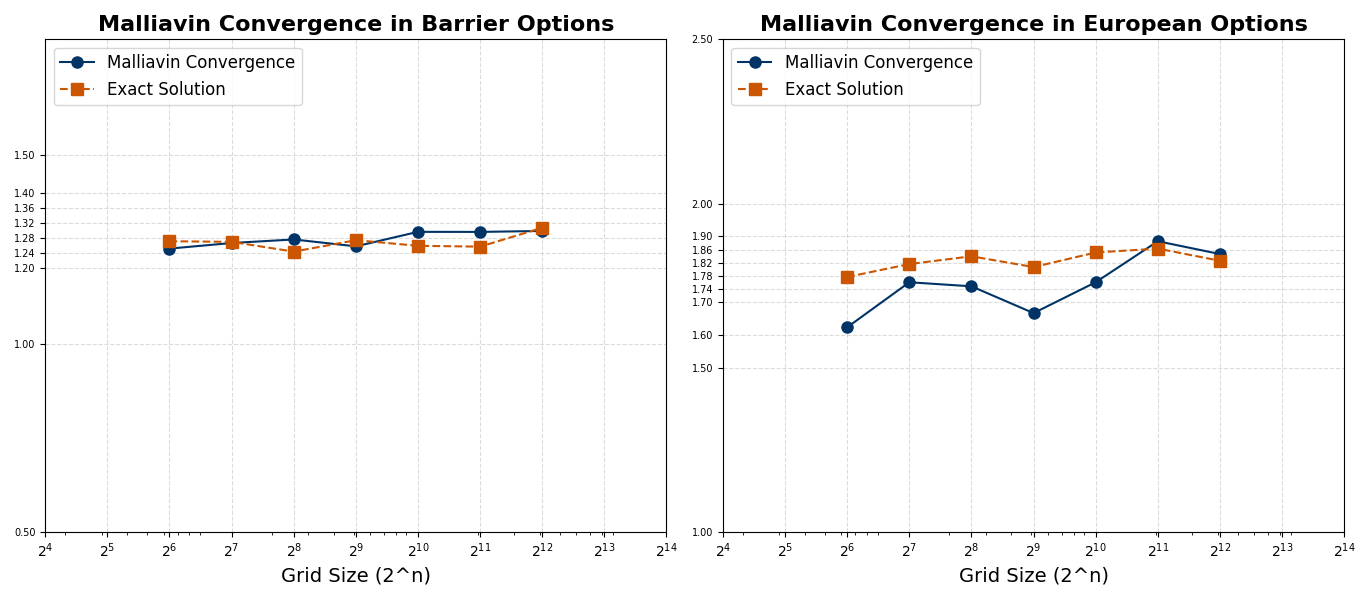}}
             \caption{\scriptsize Comparison Convergence rate of Malliavin Method with Exact solution, Example \ref{E2}}\label{fig2-EF}
\end{figure}


\newpage
 \appendix
\section{}\label{A1}
\noindent Suppose that for every $t \in [0, T]$ and $x \in L^p(P)$ the process $\tilde U_t ^{x}$, is the solution of 
\begin{align}\label{sde10}
\tilde U_t ^{x} &= x + \sum_{n=1}^{d}{\int_{0}^{t}{{\Big({\partial_1{\tilde \sigma_n}(s, S_s ^x, \pi_{S_s ^x}) \tilde U_s ^{x} + \partial_2{\tilde \sigma_n}(s, S_s ^x, \pi_{S_s ^x})\frac{\partial}{\partial x}\pi_{S_s ^x}}\Big)d{W_s ^n}}}}\notag\\
&+ \int_{0}^{t}{\int_{\mathbb{R}_0}^{}{{\Big({\partial_1{\tilde \lambda}(s, S_s ^x, z, \eta_{S_s ^x}) \tilde U_s ^{x} + \partial_3{\tilde \lambda}(s, S_s ^x, z, \eta_{S_s ^x})\frac{\partial}{\partial x}\eta_{S_s ^x}}\Big){\tilde{N} (dz, dt)}}}}.
\end{align}
We will show that this solution exists and is the stochastic flow of $S_t$.\\
Proof of Proposition 4.2:
Use Picard iteration for (\ref{sde10}), and define,
\begin{align}\label{sde4}
\tilde U_t ^{x, m} &= x + \sum_{n=1}^{d}{\int_{0}^{t}{{\Big({\partial_1{\tilde \sigma_n}(s, S_s ^x, \pi_{S_s ^x}) \tilde U_s ^{x, m-1} + \partial_2{\tilde \sigma_n}(s, S_s ^x, \pi_{S_s ^x})\frac{\partial}{\partial x}\pi_{S_s ^x}}\Big)d{W_s ^n}}}}\notag\\
&+ \int_{0}^{t}{\int_{\mathbb{R}_0}^{}{{\Big({\partial_1{\tilde \lambda}(s, S_s ^x, z, \eta_{S_s ^x}) \tilde U_s ^{x, m-1} + \partial_3{\tilde \lambda}(s, S_s ^x, z, \eta_{S_s ^x})\frac{\partial}{\partial x}\eta_{S_s ^x}}\Big){\tilde{N} (dz, dt)}}}}.
\end{align}
Applying Lemma {\ref{t1}}, we have
\begin{align*}
&\mathbb{E}\Big({\sup_{0\le s \le t} |\tilde U_s ^{x, 1} - \tilde U_s ^{x, 0}|^p}\Big) \notag\\ 
&\le \sum_{n=1}^{d}{\mathbb{E}\Big( \int_{0}^{t}{\big\{ {\partial_1{\tilde \sigma_n}(s, S_s ^x, \pi_{S_s ^x})\tilde U_s ^{x,0} + \partial_2{\tilde \sigma_n}(s, S_s ^x, \pi_{S_s ^x}) \frac{\partial}{\partial x} \pi_{S_s ^x}}\big\}^2 ds}\Big)^{p/2} } \notag\\
& + \mathbb{E} \Big( \int_{0}^{t}{\int_{\mathbb{R}_0}^{}{\big\{ {\partial_1{\tilde \lambda}(s, S_s ^x, z, \eta_{S_s ^x})\tilde U_s ^{x,0} + \partial_3{\tilde \lambda}(s, S_s ^x, z, \eta_{S_s ^x}) \frac{\partial}{\partial x} \eta_{S_s ^x}} \big\}^2}}\mu(dz)ds\Big)^{p/2}\Big\}.
\end{align*}
By Burkholder inequality can be found in [\cite{song2015regularity}, Lemma 2.3], (\ref{eq1}), (\ref{eq2}) and linear growth condition, the above last inequality is dominated by $C_p ^{\prime} |x|^p < \infty$. According to [\cite{kunita2004stochastic}, Theorem 2.11], we have further,
\begin{align*}
&\mathbb{E}\Big( \sup |\tilde U_s ^{x, m+1} - \tilde U_s ^{x, m}| ^p\Big)\notag\\
&\le \sum_{n=1}^{d}{\mathbb{E} \Big( \int_{0}^{t}{|\partial_1 {\tilde \sigma_n}(s, S_s ^x, \pi_{S_s ^x}) \tilde U_s ^{x,m} - \partial_1 {\tilde \sigma_n}(s, S_s ^x, \pi_{S_s ^x}) \tilde U_s ^{x,m-1}|^2 ds}\Big)^{p/2}}\notag\\
& + \mathbb{E} \Big( \int_{0}^{t}{\int_{\mathbb{R}_0}^{}{{|\partial_1 {\tilde \lambda}(s, S_s ^x, z, \eta_{S_s ^x}) \tilde U_s ^{x,m} - \partial_1 {\tilde \lambda}(s, S_s ^x, z, \eta_{S_s ^x}) \tilde U_s ^{x,m-1}|}^2}}\mu(dz)ds\Big)^{p/2}\notag\\
& + \mathbb{E} \Big( \int_{0}^{t}{\int_{\mathbb{R}_0}^{}{{|\partial_1 {\tilde \lambda}(s, S_s ^x, z, \eta_{S_s ^x}) \tilde U_s ^{x,m} - \partial_1 {\tilde \lambda}(s, S_s ^x, z, \eta_{S_s ^x}) \tilde U_s ^{x,m-1}|^p}\mu(dz)ds}} \Big) \Big\}.
\end{align*} 
Using (\ref{eq3}) and (\ref{eq4}) we get,
\begin{align*}
\mathbb{E} \Big(\sup_{0\le s \le t} |\tilde U_s ^{x, m+1} - \tilde U_s ^{x,m}| ^p \Big)&\le C \int_{0}^{t}{\mathbb{E} \Big( | \tilde U_s ^{x,m} - \tilde U_s ^{x, m-1}|^p\Big)ds}\notag\\
& \le C \int_{0}^{t}{\mathbb{E} \Big(\sup_{0 \le s \le t} | \tilde U_s ^{x,m} - \tilde U_s ^{x, m-1}|^p \Big)ds}.
\end{align*}
Then we get the inequality,
\begin{align}\label{eq5}
\mathbb{E}\Big( \sup_{0 \le s \le t}|\tilde U_s ^{x, m+1} &- \tilde U_s ^{x,m}| ^p \Big) \le \frac{C^m}{m!} \mathbb{E}\Big( \sup |\tilde U_s ^{x, 1} - \tilde U_s ^{x, 0}|\Big).
\end{align}
\noindent Let $m$ tend infinity in (\ref{eq5}), then
\begin{equation*}
\lim_{m \to \infty} \sup \mathbb{E} \Big( \sup_{t \in [0, T]}|\tilde U_s ^{x, m+1} - \tilde U_s ^{x,m}|^p\Big) = 0.
\end{equation*}
Therefore, there is an adapted process denoted by $\tilde U_t ^x$ such that,
\begin{equation*}
\lim_{m \to \infty} \sup \mathbb{E} \Big( \sup_{t \in [0, T]}|\tilde U_s ^{x, m} - \tilde U_s ^{x}|^p \Big) = 0.
\end{equation*}
\noindent Letting $m \to \infty$ in (\ref{sde4}), we have (\ref{sde3}). Now, we give the proof of the uniqueness. Let $\{\tilde V_t ^x\}_{t\in[0, T]}$ be a solution to (\ref{sde4}). Then we get the estimation,
\begin{equation*}
\mathbb{E}\Big( |\tilde U_t ^x - \tilde V_t ^x|^p\Big) \le C \int_{0}^{t}{\mathbb{E} \Big(|\tilde U_s ^x - \tilde V_s ^x|^p \Big)}ds,
\end{equation*}
similarly to the above argument. This implies $\mathbb{E}\Big( |\tilde U_t - \tilde V_t|^p\Big)=0.$
 \qedsymbol{}
\\
\\
Proof of Proposition \ref{flowS2}:
 From the definition of directional derivative, we have to obtain, 
\begin{equation*}
\mathbb{P}\big({\frac{\partial}{\partial x}}X_t ^x = U_t ^x, \forall t \in [0, T]\big) = 1,
\end{equation*}
according to (\ref{eq16}), it is sufficient to demonstrate,
\begin{equation*}
\mathbb{P}\big({\frac{\partial}{\partial x}}S_t ^x = \tilde U_t ^x, \forall t \in [0, T]\big) = 1.
\end{equation*}
Pursuant to (\ref{sde2}), we have
\begin{align*}
S_t ^{x+\epsilon y} = {x+\epsilon y} &+ \sum_{n=1}^{d}{\int_{0}^{t}{\tilde \sigma_n (s, S_s ^{x+\epsilon y}, \pi_{S_s ^{x+\epsilon y}})}d{W_s ^n}}\notag\\& + {\int_{0}^{t}{\int_{\mathbb{R}_0}^{}{\tilde \lambda (s, S_s ^{x+\epsilon y}, z, \eta_{S_s ^{x+\epsilon y}})}}\tilde{N}(dz, ds)}.
\end{align*}
Then,
\begin{align*}
S_t ^{x+\epsilon y} - S_t ^x &= \epsilon y + \sum_{n=1}^{d}{\int_{0}^{t}{\tilde \sigma_n (s, S_s ^{x+\epsilon y}, \pi_{S_s ^{x+\epsilon y}}) - \tilde \sigma_n (s, S_s ^x, \pi_{S_s ^x})}d{W_s ^n}} \notag\\
&+ {\int_{0}^{t}{\int_{\mathbb{R}_0}^{}{\tilde \lambda (s, S_s ^{x+\epsilon y}, z, \eta_{S_s ^{x+\epsilon y}}) - \tilde \lambda (s, S_s ^{x}, z, \eta_{S_s ^{x}})}}\tilde{N}(dz, ds)}\notag\\
& = \epsilon y + \sum_{n=1}^{d}{\int_{0}^{t}{\int_{0}^{1}{ \frac{d}{dr}\Big({\tilde \sigma_n(s, [S_s ^x + r(S_s ^{x + \epsilon y} - S_s ^x)], \pi_{S_s ^{x + \epsilon y}})}\Big)}}dr d{W_s ^n}}\notag\\
&+ \sum_{n = 1}^{d}{\int_{0}^{t}{\int_{0}^{1}{\frac{d}{dr}\Big({\tilde \sigma_n (s, S_s ^x, \pi_{S_s ^x + r(S_s ^{x+\epsilon y} - S_s ^x)})}\Big)}dr d{W_s ^n}}}\notag\\
&+ \int_{0}^{t}{\int_{\mathbb{R}_0}^{}{\int_{0}^{1}{\frac{d}{dr}\Big( {\tilde \lambda (s, [S_s ^x + r(S_s ^{x + \epsilon y} - S_s ^x)], z, \eta_{S_s ^{x + \epsilon y}})}\Big)}dr \tilde{N}(dz, ds)}}\notag\\
&+ \int_{0}^{t}{\int_{\mathbb{R}_0}^{}{\int_{0}^{1}{\frac{d}{dr}\Big({\tilde \lambda (s, S_s ^x, z, \eta_{S_s ^x +r(S_s ^{x + \epsilon y} - S_s ^x)})}\Big)}dr \tilde{N}(dz, ds)}}\notag\\
& = \epsilon y + \sum_{n=1}^{d}{\int_{0}^{t}{\alpha_{1n}(s)(S_s ^{x + \epsilon y} - S_s ^x)}d{W_s ^n}} + \sum_{n=1}^{d}{\int_{0}^{t}{\beta_{1n}(s)}dW_s ^n}\notag\\
&+ \int_{0}^{t}{\int_{\mathbb{R}_0}^{}{\alpha_2 (s)(S_s ^{x + \epsilon y} - S_s ^x)\tilde{N}(dz, ds)}} + \int_{0}^{t}{\int_{\mathbb{R}_0}^{}{\beta_2}(s)\tilde{N}(dz, ds)},
\end{align*}
where
\begin{align*}
&\alpha_{1n} = \int_{0}^{1}{\partial_1 \tilde \sigma_n(s, [S_s ^x + r(S_s ^{x + \epsilon y} - S_s ^x)], \pi_{S_s ^{x + \epsilon y}})}dr,\notag\\
&\alpha_2 = \int_{0}^{1}{\partial_1 \tilde \lambda((s, [S_s ^x + r(S_s ^{x + \epsilon y} - S_s ^x)], z, \eta_{S_s ^{x + \epsilon y}}))}dr,\notag\\
&\beta_{1n} = \int_{0}^{1}{\partial_2 \tilde \sigma_n(s, S_s ^x, \pi_{S_s ^x + r(S_s ^{x + \epsilon y} - S_s ^x)})\mathbb{E}(S_s ^{x + \epsilon y} - S_s ^x)}dr,\notag\\
&\beta_2 = \int_{0}^{1}{\partial_3 \tilde \lambda(s, S_s ^x, z, \eta_{S_s ^x + r(S_s ^{x + \epsilon y} - S_s ^x)})\mathbb{E}(S_s ^{x + \epsilon y} - S_s ^x)}dr.
\end{align*}
\noindent For any $p \ge 2$ from Lemma \ref{t1} and Assumption 4.1, it holds
\begin{align*}
\mathbb{E}&\Big( {\sup_{t \in [0, T]} |S_s ^{x + \epsilon y} - S_s ^x|^p} \Big) \notag\\
&\le C_p {\Big{\{}}|x|^p \epsilon^p 
+ \sum_{n=1}^{d}{\mathbb{E}\Big(\big({\int_{0}^{t}{|\alpha_{1n}(s)|^2 |S_s ^{x + \epsilon y} - S_s ^x|^2}}ds\big)^{p/2}\Big)} \notag\\
&+ \sum_{n=1}^{d}{\mathbb{E}\Big(\big({\int_{0}^{t}{|\beta_{1n}(s)|^2}}ds\big)^{p/2}\Big)} \notag\\
&+ \mathbb{E}\Big(\big({\int_{0}^{t}{\int_{\mathbb{R}_0}^{}{|\alpha_2(s)|^2 |S_s ^{x+\epsilon y} - S_s ^x|^2 \mu(dz)ds}}}\big)^{p/2}\Big)\notag\\
&+\mathbb{E}{\Big(\big({\int_{0}^{t}{\int_{\mathbb{R}_0}^{}{|\beta_2(s)|^2}}}\big)^{p/2}\Big)}+ \mathbb{E}\Big({\int_{0}^{t}{\int_{\mathbb{R}_0}^{}{|\alpha_2|^p |S_s ^{x + \epsilon y} - S_s ^x|^p \mu(dz)ds}}}\Big){\Big{\}}} \notag\\
& \le C_p |y|^p \epsilon^p + C_p \int_{0}^{1}{\mathbb{E}\Big({\sup_{s \in [0, T]}|S_s ^{x + \epsilon y} - S_s ^x|^p}\Big)ds}.
\end{align*}
Hence,
\begin{equation}\label{eq20}
\mathbb{E}\Big({{\sup_{t \in [0, T]}|S_t ^{x + \epsilon y} - S_t ^x|^p}}\Big) \le C_p \mathbb{E}|y|^p \epsilon^p.
\end{equation}
Observe that,
\begin{align*}
S_t ^{x+\epsilon x} - S_t ^x - \epsilon {\tilde U}_t ^x &= \sum_{n=1}^{d}{\int_{0}^{t}{\partial_1 \tilde \sigma_n(s, S_s ^x, \pi_{S_s ^x})\big[{S_s ^{x+\epsilon x} - S_s ^x - \epsilon {\tilde U}_s ^x}\big]dW_s ^n}}\notag\\
& +  \sum_{n=1}^{d}{\int_{0}^{t}{{\beta}_{1n} - \epsilon \partial_2 \tilde \sigma_n(s, S_s ^x, \pi_{S_s ^x}) {\frac{\partial}{\partial x} \pi_{S_s ^x}}dW_s ^n}}\notag\\
&  + {\int_{0}^{t}{\int_{\mathbb{R}_0}^{}{\partial_1 \tilde \lambda(s, S_s ^x, z, \eta_{S_s ^x})\big[{S_s ^{x+\epsilon x} - S_s ^x - \epsilon {\tilde U}_s ^x}\big]\tilde{N}(dz, ds)}}}\notag\\
& +{\int_{0}^{t}{\int_{\mathbb{R}_0}^{}{\beta_2 - \epsilon\partial_3 \tilde \lambda(s, S_s ^x, z, \eta_{S_s ^x}){\frac{\partial}{\partial x}\eta_{S_s ^x}}\tilde{N}(dz, ds)}}}\notag\\
& + \sum_{n=1}^{d}{\int_{0}^{t}{\big({\alpha_{1n} - \partial_1 \tilde \sigma_n(s, S_s ^x, \pi_{S_s ^x})}\big)\big({S_s ^{x+\epsilon x} - S_s ^x}\big)}dW_s ^n}\notag\\
& + \int_{0}^{t}{\int_{\mathbb{R}_0}^{}{\big({\alpha_{2} - \partial_1 \tilde \lambda(s, S_s ^x, z, \eta_{S_s ^x})}\big)\big({S_s ^{x+\epsilon x} - S_s ^x}\big)}\tilde{N}(dz, ds)}.\notag\\
\end{align*}
\noindent Then it follows from Lemma \ref{t1} and Assumption 4.1,
\begin{align*}
\mathbb{E} &\Big({\sup_{t \in [0, T]} |S_t ^{x+\epsilon x} - S_t ^x - \epsilon {\tilde U}_t ^x|^2} \Big)\notag\\
&\le C \int_{0}^{1}{\mathbb{E}\Big({\sup_{s \in [0, T]}|S_s ^{x+\epsilon x} - S_s ^x - \epsilon {\tilde U}_s ^x|^2}\Big) ds}+ C\mathbb{E}\Big({\sup_{s \in [0, T]}|S_t ^{x+\epsilon x} - S_t ^x |^4}\Big).
\end{align*}
By Granwall inequality and (\ref{eq20}),
\begin{align*}
\lim_{\epsilon \to 0}\mathbb{E}\Big({\sup_{t \in [0, T} |S_t ^{x+\epsilon x} - S_t ^x - \epsilon {\tilde U}_t ^x|}\Big)^2 &\le C \lim_{\epsilon \to 0}\mathbb{E}\Big({\sup_{s \in [0, T]}|S_s ^{x+\epsilon x} - S_s ^x|^4}\Big) =0.
\end{align*}
From the definition of directional derivative, we obtain,
\begin{equation*}
\mathbb{P}\big({\frac{\partial}{\partial x}}S_t ^x = \tilde U_t ^x, \forall t \in [0, T]\big) = 1.
\end{equation*}
The proof has been completed.
\qedsymbol{}

\begin{thebibliography}{9}  

\bibitem{fournie1999applications}  
E. Fournié, J.-M. Lasry, J. Lebuchoux, P.-L. Lions, and N. Touzi, \textit{Applications of Malliavin calculus to Monte Carlo methods in finance}, Finance \& Stochastics, vol. 3, no. 4, pp. 391–412, 1999.  

\bibitem{el2004computations}  
Y. El-Khatib and N. Privault, \textit{Computations of Greeks in a market with jumps via the Malliavin calculus}, Finance and Stochastics, vol. 8, no. 2, pp. 161–179, 2004.  

\bibitem{davis2006malliavin}  
M. H. Davis and M. P. Johansson, \textit{Malliavin Monte Carlo Greeks for jump diffusions}, Stochastic Processes and their Applications, vol. 116, no. 1, pp. 101–129, 2006.  

\bibitem{bavouzet2006computation}  
M. P. Bavouzet and M. Messaoud, \textit{Computation of Greeks using Malliavin’s calculus in jump type market models}, Electronic Journal of Probability, vol. 11, no. 10, pp. 276–300, 2006.  

\bibitem{forster2009absolutely}  
B. Forster, E. Lütkebohmert, and J. Teichmann, \textit{Absolutely continuous laws of jump-diffusions in finite and infinite dimensions with applications to mathematical finance}, SIAM Journal on Mathematical Analysis, vol. 40, no. 5, pp. 2132–2153, 2009.  

\bibitem{petrou2008malliavin}  
E. Petrou, \textit{Malliavin calculus in Lévy spaces and applications to finance}, Electronic Journal of Probability, vol. 13, no. 27, pp. 852–879, 2008.  

\bibitem{benth2011robustness}  
F. E. Benth, G. Di Nunno, and A. Khedher, \textit{Robustness of option prices and their deltas in markets modelled by jump-diffusions}, Communications on Stochastic Analysis, vol. 5, no. 2, pp. 285–307, 2011.  

\bibitem{khedher2012computation}  
A. Khedher, \textit{Computation of the Delta in multidimensional jump-diffusion setting with applications to stochastic volatility models}, Stochastic Analysis and Applications, vol. 30, no. 3, pp. 403–425, 2012.  

\bibitem{hudde2023european}  
A. Hudde and L. Rüschendorf, \textit{European and Asian Greeks for Exponential Lévy Processes}, Methodology and Computing in Applied Probability, vol. 25, no. 1, pp. 1–24, 2023.  

\bibitem{yilmaz2018computation}  
B. Yilmaz, \textit{Computation of option Greeks under hybrid stochastic volatility models via Malliavin calculus}, Modern Stochastics: Theory and Applications, vol. 5, no. 2, pp. 145–165, 2018.  

\bibitem{benth2021sensitivity}  
F. E. Benth, G. Di Nunno, and I. C. Simonsen, \textit{Sensitivity analysis in the infinite dimensional Heston model}, Infinite Dimensional Analysis, Quantum Probability and Related Topics, vol. 24, no. 02, pp. 2150014–66, 2021.  

\bibitem{fan2022bismut}  
X. Fan and R. Yu, \textit{Bismut type derivative formulae and gradient estimate for multiplicative SDEs with fractional noises}, Stochastics, vol. 94, no. 4, pp. 493–518, 2022.  

\bibitem{coffie2023sensitivity}  
E. Coffie, S. Duedahl, and F. Proske, \textit{Sensitivity analysis with respect to a stochastic stock price model with rough volatility via a Bismut–Elworthy–Li formula for singular SDEs}, Stochastic Processes and their Applications, vol. 156, pp. 156–195, 2023.  

\bibitem{kac1956foundations}  
M. Kac, \textit{Foundations of kinetic theory}, University of California Press, 1956.  

\bibitem{mckean1966class}  
H. P. McKean Jr, \textit{A class of Markov processes associated with nonlinear parabolic equations}, Proceedings of the National Academy of Sciences, vol. 56, no. 6, pp. 1907–1911, 1966.  

\bibitem{graham1992mckean}  
C. Graham, \textit{McKean-Vlasov Itô-Skorohod equations, and nonlinear diffusions with discrete jump sets}, Stochastic Processes and their Applications, vol. 40, no. 1, pp. 69–82, 1992.  


\bibitem{carmona2018probabilistic}  
R. Carmona and F. Delarue, \textit{Probabilistic theory of mean-field games with applications I-II}, Springer, 2018.  

\bibitem{bush2011stochastic}  
N. Bush, B. M. Hambly, H. Haworth, L. Jin, and C. Reisinger, \textit{Stochastic evolution equations in portfolio credit modelling}, SIAM Journal on Financial Mathematics, vol. 2, no. 1, pp. 627–664, 2011.  

\bibitem{hambly2019spde}  
B. Hambly and A. Søjmark, \textit{An SPDE model for systemic risk with endogenous contagion}, Finance and Stochastics, vol. 23, no. 3, pp. 535–594, 2019.  

\bibitem{dawson1995stochastic}  
D. Dawson and J. Vaillancourt, \textit{Stochastic McKean-Vlasov equations}, Nonlinear Differential Equations and Applications NoDEA, vol. 2, no. 2, pp. 199–229, 1995.  

\bibitem{banos2018bismut}  
D. Baños, \textit{The Bismut-Elworthy-Li formula for mean-field stochastic differential equations}, vol. 54, no. 1, pp. 220–233, 2018.  

\bibitem{bao2021bismut}  
J. Bao, P. Ren, and F.-Y. Wang, \textit{Bismut formula for Lions derivative of distribution-path dependent SDEs}, Journal of Differential Equations, vol. 282, pp. 285–329, 2021.  

\bibitem{song2022regularity}  
Y. Song and Z. Wang, \textit{Regularity for distribution-dependent SDEs driven by jump processes}, Stochastics and Dynamics, vol. 22, no. 5, pp. 2250011–35, 2022.  

\bibitem{nualart2018introduction}  
D. Nualart and E. Nualart, \textit{Introduction to Malliavin calculus}, Cambridge University Press, 2018.  

\bibitem{tahmasebi2022bismut}  
M. Tahmasebi, \textit{The Bismut-Elworthy-Li formula for semi-linear distribution-dependent SDEs driven by fractional Brownian motion}, arXiv preprint arXiv:2209.05586, 2022.  

\bibitem{agram2024optimal}  
N. Agram and B. Øksendal, \textit{Optimal stopping of conditional McKean-Vlasov jump diffusions}, Systems \& Control Letters, vol. 188, p. 105815, 2024.  

\bibitem{agram2024deep}  
N. Agram and J. Rems, \textit{Deep learning for conditional McKean-Vlasov jump diffusions}, Available at SSRN, 2024.  


\bibitem{li2016weak}  
J. Li and H. Min, \textit{Weak solutions of mean-field stochastic differential equations and application to zero-sum stochastic differential games}, SIAM Journal on Control and Optimization, vol. 54, no. 3, pp. 1826–1858, 2016.  

\bibitem{bauer2018strong}  
M. Bauer, T. Meyer-Brandis, and F. Proske, \textit{Strong solutions of mean-field stochastic differential equations with irregular drift}, Electronic Journal of Probability, vol. 23, no. 132, pp. 1–35, 2018.  

\bibitem{mishura2020existence}  
Y. Mishura and A. Veretennikov, \textit{Existence and uniqueness theorems for solutions of McKean–Vlasov stochastic equations}, Theory of Probability and Mathematical Statistics, vol. 103, pp. 59–101, 2020.  

\bibitem{erny2022well}  
X. Erny, \textit{Well-posedness and propagation of chaos for McKean–Vlasov equations with jumps and locally Lipschitz coefficients}, Stochastic Processes and their Applications, vol. 150, pp. 192–214, 2022.  

\bibitem{liu2023onsager}  
S. Liu, H. Gao, H. Qiao, and N. Lu, \textit{The Onsager-Machlup action functional for McKean-Vlasov stochastic differential equations}, Communications in Nonlinear Science and Numerical Simulation, vol. 121, pp. 107203–18, 2023.  

\bibitem{ning2024one}  
N. Ning, J. Wu, and J. Zheng, \textit{One-dimensional McKean–Vlasov stochastic variational inequalities and coupled BSDEs with locally Hölder noise coefficients}, Stochastic Processes and their Applications, vol. 171, no. C, pp. 104315–62, 2024.  

\bibitem{chen2024correspondence}  
J. Chen, J. Hu, Z. Wang, and T. G. J. Duan, \textit{Correspondence Research of the Most Probable Transition Paths between a Stochastic Interacting Particle System and its Mean Field Limit System}, arXiv preprint arXiv:2404.07552, 2024.  

\bibitem{buckdahn2017mean}  
R. Buckdahn and S. Jing, \textit{Mean-field SDE driven by a fractional Brownian motion and related stochastic control problem}, SIAM Journal on Control and Optimization, vol. 55, no. 3, pp. 1500–1533, 2017.  

\bibitem{sun2017ito}  
Y. Sun, J. Yang, and W. Zhao, \textit{Ito-Taylor schemes for solving mean-field stochastic differential equations}, Numerical Mathematics: Theory, Methods and Applications, vol. 10, no. 4, pp. 798–828, 2017.  

\bibitem{sun2020explicit}  
Y. Sun and W. Zhao, \textit{An explicit second-order numerical scheme for mean-field forward backward stochastic differential equations}, Numerical Algorithms, vol. 84, no. 1, pp. 253–283, 2020.  

\bibitem{sun2021numerical}  
Y. Sun and W. Zhao, \textit{Numerical methods for mean-field stochastic differential equations with jumps}, Numerical Algorithms, vol. 88, no. 2, pp. 1–35, 2021.  

\bibitem{liu2023tamed}  
H. Liu, B. Shi, and F. Wu, \textit{Tamed Euler–Maruyama approximation of McKean–Vlasov stochastic differential equations with super-linear drift and Hölder diffusion coefficients}, Applied Numerical Mathematics, vol. 183, pp. 56–85, 2023.  

\bibitem{bao2022approximations}  
J. Bao and X. Huang, \textit{Approximations of McKean–Vlasov stochastic differential equations with irregular coefficients}, Journal of Theoretical Probability, vol. 35, no. 2, pp. 1–29, 2022.  

\bibitem{li2023strong}  
Y. Li, X. Mao, Q. Song, F. Wu, and G. Yin, \textit{Strong convergence of Euler–Maruyama schemes for McKean–Vlasov stochastic differential equations under local Lipschitz conditions of state variables}, IMA Journal of Numerical Analysis, vol. 43, no. 2, pp. 1001–1035, 2023.  

\bibitem{higham2005convergence}  
D. Higham and X. Mao, \textit{Convergence of Monte Carlo simulations involving the mean-reverting square root process}, The Journal of Computational Finance, vol. 8, no. 3, pp. 35–61, 2005.  

\bibitem{nunno2008malliavin}  
G. D. Nunno, B. Øksendal, and F. Proske, \textit{Malliavin calculus for Lévy processes with applications to finance}, Springer, 2008.  

\bibitem{hao2016mean}  
T. Hao and J. Li, \textit{Mean-field SDEs with jumps and nonlocal integral-PDEs}, Nonlinear Differential Equations and Applications NoDEA, vol. 23, no. 2, pp. 1–51, 2016.  

\bibitem{kunita2004stochastic}  
H. Kunita, \textit{Stochastic differential equations based on Lévy processes and stochastic flows of diffeomorphisms}, Real and Stochastic Analysis: New Perspectives, pp. 305–373, 2004.  

\bibitem{menaldi2008stochastic}  
J.-L. Menaldi, \textit{Stochastic differential equations with jumps}, Mathematics Faculty Research Publications, 2008.  



\bibitem{huehne2005malliavin}  
F. Huehne, \textit{Malliavin calculus for the computation of Greeks in markets driven by pure-jump Lévy processes}, 2005. Available at: http://dx.doi.org/10.2139/ssrn.948347.  

\bibitem{schoutens2003pricing}  
W. Schoutens and S. Symens, \textit{The pricing of exotic options by Monte-Carlo simulations in a Lévy market with stochastic volatility}, International Journal of Theoretical and Applied Finance, vol. 6, no. 08, pp. 839–864, 2003.  

\bibitem{platen2010numerical}  
E. Platen and N. Bruti-Liberati, \textit{Numerical solution of stochastic differential equations with jumps in finance}, vol. 64. Springer Science \& Business Media, 2010.  

\bibitem{song2015regularity}  
Y. Song and X. Zhang, \textit{Regularity of density for SDEs driven by degenerate Lévy noises}, Electron Journal of Probability, vol. 20, no. 21, pp. 1–27, 2015.  

\end{thebibliography}
\end{document}